\documentclass[10pt,a4paper]{article}
\usepackage[T1]{fontenc}

\usepackage{fullpage}

\usepackage{amsmath,amssymb,amsthm,graphicx}
\usepackage[subrefformat=parens,labelformat=parens]{subfig}
\usepackage{float}
\usepackage{verbatim}
\usepackage[dvipsnames]{xcolor}
\usepackage{tikz}
\usetikzlibrary{positioning,shapes,shadows,calc}
\usetikzlibrary{arrows,automata}

\usepackage[shortlabels]{enumitem}

\let\emph\relax
\DeclareTextFontCommand{\emph}{\bfseries}

\usepackage[sort]{cite}

\usetikzlibrary{decorations.pathreplacing,shapes.multipart}

\makeatletter
\DeclareRobustCommand{\rvdots}{%
	\vbox{
		\baselineskip4\p@\lineskiplimit\z@
		\kern-\p@
		\hbox{.}\hbox{.}\hbox{.}
}}
\makeatother

\newcommand\at{\text{{\fontfamily{qbk}\selectfont\small @}}}
\newcommand\acts{\cdot}

\setlist{nosep}

\usepackage{mathtools, stmaryrd}

\usepackage{microtype}
\usepackage{xcolor}
\definecolor{darkblue}{rgb}{0.0, 0.0, 0.55}
\usepackage[hidelinks]{hyperref}
\hypersetup{%
	colorlinks=true,%
	linkcolor = darkblue,
	anchorcolor = darkblue,
	citecolor = darkblue,
	filecolor = darkblue,
	urlcolor = darkblue,
	linktoc=page
}

\usepackage[capitalise,noabbrev,sort]{cleveref}

\usepackage{thmtools, thm-restate}

\crefalias{theoremx}{theorem}

\newcommand{\Td}{\mathcal{T}_d}
\newcommand{\AutTd}{\mathrm{Aut}(\Td)}
\newcommand{\FinTd}{\mathrm{Fin}(\Td)}
\newcommand{\DirTd}{\mathrm{Dir}(\Td)}
\newcommand{\BTd}{\mathcal{B}(\Td)}
\newcommand{\WP}{\mathsf{WP}}

\renewcommand{\leq}{\leqslant}
\renewcommand{\geq}{\geqslant}

\newcommand\CompDepth{\mathrm{DDepth}}

\newcommand\Decorate{\mathrm{Dec}}

\newcommand\Reg{\mathrm{Reg}}

\newtheorem{theorem}{Theorem}[section]
\newtheorem{proposition}[theorem]{Proposition}
\newtheorem{lemma}[theorem]{Lemma}

\newtheorem{definition}[theorem]{Definition}
\newtheorem*{definition*}{Definition}

\newtheorem{question}[theorem]{Question}

\newtheorem{remark}[theorem]{Remark}
\newtheorem{notation}[theorem]{Notation}

\renewcommand{\star}{^\ast}

\newcommand{\N}{\mathbb{N}}

\newcommand{\Stab}{\mathrm{Stab}}

\makeatletter
\newcommand{\IfRestatedTF}[2]{\ifthmt@thisistheone #2\else #1\fi}
\makeatother

\tikzset{initial text={}}
\makeatletter
\tikzstyle{initial by arrow}=   [after node path=
{
  {
    [to path=
    {
      [->,double=none,every initial by arrow]
      ([shift=(\tikz@initial@angle:\tikz@initial@distance)]\tikztostart.\tikz@initial@angle)
          node [shape=rectangle,anchor=\tikz@initial@anchor] {\tikz@initial@text}
        -- (\tikztostart)}]
    edge ()
  }
}]
\makeatother

\begin{document}

\title{On the ET0L subgroup membership problem in bounded automata groups}
\date{}
\author{Alex Bishop%
\footnote{Université de Genève, Genève, Switzerland
\texttt{alexbishop1234@gmail.com}}%
\and Daniele D'Angeli %
\footnote{Universit\`a Niccolo Cusano, Rome, Italy \texttt{daniele.dangeli@unicusano.it}}
\and Francesco Matucci%
\footnote{Universit\`{a} %
di Milano--Bicocca, Milan, Italy.
\texttt{francesco.matucci@unimib.it}}
\and
Tatiana Nagnibeda%
\footnote{Université de Genève, Genève, Switzerland
\texttt{tatiana.smirnova-nagnibeda@unige.ch}}
\and Davide Perego%
\footnote{Université de Genève, Genève, Switzerland
\texttt{davide.perego@unige.ch}}
\and Emanuele Rodaro%
\footnote{Politecnico di Milano, Milan, Italy.
\texttt{emanuele.rodaro@polimi.it}}}

\maketitle

\begin{abstract}
We are interested in the subgroup membership problem in groups acting on rooted $d$-regular trees and a natural class of subgroups, the stabilisers of infinite rays emanating from the root. These rays, which can also be viewed as infinite words in the alphabet with $d$ letters, form the boundary of the tree. Stabilisers of infinite rays are not finitely generated in general, but if the ray is computable, the membership problem is well posed and solvable. The main result of the paper is that, for bounded automata groups, the membership problem in the stabiliser of any ray that is eventually periodic as an infinite word, forms an ET0L language that is constructable.
The result is optimal in the sense that, in general, the membership problem for the stabiliser of an infinite ray in a bounded automata group cannot be context-free. As an application, we give a recursive formula for the associated generating function, aka the Green function, on the corresponding infinite Schreier graph.
\end{abstract}

\section{Introduction}\label{sec:introduction}

In an influential paper in 1911 \cite{dehn1911}, Max Dehn formulated three decision problems for finitely generated groups, the most famous of them being the \emph{word problem}.
For a finitely generated group $G$ with a finite symmetric generating set $X$, the word problem asks if we can decide, given a word $w\in X^*$ in the free monoid over the alphabet $X$, whether $\overline{w}$, the natural projection of the word $w$ to the group is the trivial element (in other words, belongs to the trivial subgroup).
A natural extension of this problem is the \emph{subgroup membership problem}, which asks, given a word $w\in X^*$ and a description of a subgroup $H\leq G$, whether $\overline{w}$ is an element of $H$.
For $G, X$ and $H\leq G$, denote by $\WP(G,X,H)$ the set of all words $w\in X^*$ for which $\overline{w}$ is an element of the subgroup $H$.
We can then ask for which classes of finitely generated groups and subgroups the membership to $\WP(G,X,H)$ is uniformly computationally decidable.  Of course, for the problem to be well posed one needs the subgroups in question to have a computable description. The most popular case is to consider the membership problem in finitely generated subgroups, which, given finitely many elements $g_1,g_2,...,g_k\in G$ and a word $w\in X^*$, asks whether $\overline{w}$ belongs to the subgroup  $H = \left\langle g_1,g_2,...,g_k \right\rangle$.
Decidability of the subgroup membership problem (mostly for finitely generated subgroups) for various classes of groups has been studied over the years, and we refer the reader to the recent survey \cite{Lohrey} for an excellent account of the state-of-the-art of the subject. 

In this paper, we focus on one interesting aspect of the subgroup membership problem, that is, to describe the \emph{formal language} $\WP(G,X,H)$ for given $G$, $X$ and $H$. The word problem $\WP(G,X)$ can be thought of as the language formed by the words read along the closed paths in the Cayley graph of $(G, X)$ where the edges are oriented and labelled by letters from $X$ and the paths are based at the vertex representing the identity element. Similarly, the 
subgroup membership problem $\WP(G,X,H)$ consists of the words read along the closed paths in the Schreier graph of $(G, X, H)$ based at the vertex representing the trivial coset $H$ in the Schreier graph. Such languages have been studied extensively for the case when H is trivial (the word
problem), but practically nothing is known for the case of non-trivial H.
Anisimov proved in \cite{anisimov1971} that the word problem constitutes a regular language if and only if the group is finite, and this result readily generalises to arbitrary $H$: the language $\WP(G,X,H)$ is regular if and only if $H$ is a subgroup of finite index~\cite[Proposition~6.1]{sakarovitch2009elements}.
  The famous theorem of Muller and Schupp  \cite{MullerSchupp} tells us that the word problem is a context-free language if and only if the group is virtually free. 
  
  Regular and context-free languages constitute two smallest classes in Chomsky's hierarchy of formal languages.  The next class is that of context-sensitive languages, but in recent years other, intermediate classes  came into play. For example, the \emph{ET0L languages}  introduced by Rozenberg in his 1973 paper  \cite{Rozenberg1973} recently became popular in geometric group theory (see, e.g., \cite{ciobanu2018,Bishop2019,diekert2017,ciobanu2019,evetts2020,diekert2001}).  Recall that the regular (respectively, context-free) languages can be characterised as those that are recognised by finite-state (respectively, pushdown) automata. Analogously, ET0L languages are exactly those that can be recognised by a check-stack push-down automaton~\cite{leeuwen1976}.

  It is conjectured (Conjecture~8.1 in~\cite{ciobanu2018}) that a group has an ET0L word problem if and only if it is virtually free. The main aim of this paper is to present a family of groups and subgroups where the class of subgroup membership problem is exactly ET0L.

We consider  \emph{groups generated by automorphisms of  regular rooted trees}. For a fixed integer $d\geq 2$, we write $\Td$ for the $d$-regular rooted tree. The vertices of $\Td$ can be identified with words in the free monoid $C^*$ in an alphabet $C = \{c_1,c_2,...,c_d\}$ with $d$ letters. The one-sided infinite words in the alphabet $C$ represent infinite rays emanating from the root that form the boundary $\partial\mathcal{T}_d$ of the tree. 
 Let $G \leq \AutTd$ be a finitely generated group of automorphisms of $\Td$. Its action on the tree by automorphisms extends by continuity to an action  on the boundary of the tree by homeomorphisms.
 
Among the subgroups of $G$, a special role is played by \emph{point stabilisers} for this action. The stabilisers of the tree vertices are subgroups of finite index and hence the subgroup membership problem in them is a regular language. From now on we concentrate on the membership problem in stabilisers of infinite rays. We denote by $\Stab(\eta)$ the stabiliser of the infinite ray $\eta = c_{i_1} c_{i_2} c_{i_3} \cdots \in C^\omega$, that is,
\begin{equation}\tag{$\diamond$}\label{eq:stab}
    \Stab(\eta) = \bigcap_{k=1}^\infty \Stab(c_{i_1} c_{i_2}\cdots c_{i_k})
\end{equation}
where each $\Stab(c_{i_1} c_{i_2} \cdots c_{i_k})$ is the stabiliser in $G$ of the vertex of the tree $\Td$ corresponding to the word $c_{i_1}c_{i_2}\cdots c_{i_k}$.

We focus on \emph{automaton automorphisms} of $\Td$ (see \cref{def:automata-automorphism}) which can be completely described by a finite amount of data.

This allows us to study the subgroup membership problem uniformly over the class of groups generated by finitely many automata automorphisms. This is a very interesting class of groups that includes many important examples, such as groups of intermediate growth, infinite torsion groups, non-elementary amenable groups, and more \cite{Bondarenko2007thesis,Nekrashevych2005book}.

 We now turn to the subgroup membership problem to stabilisers of infinite rays in finitely generated automata groups. These stabiliser subgroups are not finitely generated in many interesting examples, but we can use the infinite ray $\eta$ as an input of our algorithmic problem, and we require $\eta$ to be computable.

To find groups and subgroups with ET0L subgroup membership problem, we further specialise to groups  generated by \emph{bounded automaton automorphisms}\footnote{Remark: given an automaton automorphism, it is computable to check if it is bounded} (see \cref{def:bounded-automata-automorphism}) and stabilisers of rays that are  \emph{eventually periodic}, that is, of the form $\eta = ab^\omega$ with $a,b\in C^*$.

In the main result of the paper, \cref{thm:main}, we show that given a finite set of bounded automaton automorphisms $X$ of the rooted tree $\Td = C^*$, and words $a,b\in C^*$,  the subgroup membership problem $\WP(G,X,\Stab(\eta))$, with $G = \left\langle X \right\rangle$ and $\eta=ab^\omega$, is an ET0L language as described in \cref{sec:et0l-language}.
In particular, we show that we can effectively compute a description of such an ET0L language by an \emph{unambiguous limiting ET0L grammar} (see \cref{def:limiting-et0l}).
Describing an ET0L language in this way then enables us to apply \cref{thm:gfun} to find a description of its generating function.

\begin{restatable*}{theorem}{TheoremMain}\label{thm:main}
Suppose that we are given a finite symmetric set $X$ of bounded automaton automorphisms acting on the tree $\Td = C\star$, and words $a,b\in C^*$.
Then, we can effectively compute---uniformly over all $X$, $a$ and $b$---an ET0L grammar which generates the language $\WP(G,X,\Stab(\eta))$ with $G = \left\langle X \right\rangle$ and $\eta=ab^\omega$, and an ET0L grammar which generates the complement of this language, i.e., $X^*\setminus \WP(G,X,\Stab(\eta))$.
Moreover, in both cases the grammars are unambiguous limiting.
\end{restatable*}

In language theory, a frequent question is where a specific class of formal languages fits within the Chomsky hierarchy. This hierarchy is a way to classify formal grammars and the languages they generate based on their expressive power. Formally, expressive power can be understood as the ability of the grammar to enforce increasingly complex dependencies between symbols in a string — from local constraints in regular languages, to nested structures in context-free languages, and ultimately to arbitrary computable relations in recursively enumerable languages.
It is well known that the class of ET0L languages lies strictly between the families of context-free and context-sensitive languages (see Theorem~19 of \cite{Rozenberg1973}).
It is hence natural to ask whether our main result can be strengthened and whether the membership problem into stabilisers of infinite periodic rays belongs in fact to the class of context-free languages.
In \cref{sec:not-CF}, we work with Schreier graphs of the corresponding subgroups to provide obstructions to the membership problem being context-free.
In particular, we show that any stabiliser of an infinite ray $\eta$ in the first Grigorchuk group and other key examples of bounded automata groups has a non-context-free language $\WP(G,X,\Stab(\eta))$. This demonstrates that  \cref{thm:main} cannot be improved to the class of context-free languages.

The fact that the membership problems described in \cref{thm:main} are unambiguous limiting ET0L languages implies that their generating functions are computable. More precisely, we  compute a recurrence relation for the generating functions of such unambiguous limiting ET0L languages in \cref{thm:gfun}.
For a background on the notation used in \cref{thm:gfun}, see \cref{sec:gfun}.

\begin{restatable*}{theorem}{TheoremGFun}\label{thm:gfun}
Let $L\subseteq \Sigma\star$ be an unambiguous limiting ET0L language.
Then, it is computable to find a description of the generating function of $L$ as
\[
    f(z) = g(r_1(z), r_2(z),...,r_k(z))
\]
where each $r_i(z)\in \mathbb{N}[[z]]$ is a rational power series, and $g(x_1,x_2,...,x_k)$ is a formal power series defined as
\[
    g(x_1,x_2,...,x_k)
    =
    \lim_{n\to\infty} g_n(x_1,x_2,...,x_k)
\]
where $g_0(x_1,x_2,...,x_k)\in \mathbb N[[x_1,...,x_k]]$ is a rational power series, and
\[
    g_{n+1}(x_1,x_2,...,x_k) = g_{n}(q_1(x_1,x_2,...,x_k),q_2(x_1,x_2,...,x_k),
    ...,q_k(x_1,x_2,...,x_k))
\]
for each $n \geq 0$ where each $q_i(x_1,...,x_k) \in \mathbb{N}[[x_1,...,x_k]]$ is a rational series depending on $g_n$. In the above $k$ is a constant that depends on the grammar.
\end{restatable*}

Suppose that $X$ is a finite symmetric set of automaton automorphisms which generates the group $G = \left\langle X \right\rangle$, and that $\eta \in \partial \Td$,
then the \emph{Schreier graph} $\Gamma_\eta$ of the subgroup $\Stab(\eta)$  is formed by the vertex set $\eta\cdot G = \{\eta\cdot g \mid g\in G\}$ with a labelled edge $\zeta\to^x (\zeta\acts x)$ for each vertex $\zeta$ and each generator $x\in X$.

\Cref{thm:gfun} can be used to calculate the generating functions of closed walks on such Schreier graphs.
Indeed, \cref{thm:gfun} allows us to find a recurrence for the generating function $f(z)$ of the language $\mathrm{WP}(G,X,\mathrm{Stab}(\eta))$, as in \cref{thm:main}. As mentioned above, $\mathrm{WP}(G,X,\mathrm{Stab}(\eta))$ consists exactly of  all words read along the  paths in $\Gamma_\eta$ that begin and end at the vertex $\eta$.
This generating function is closely related to the Green function of the simple random walk on $\Gamma_\eta$:
\[
    G(z)
    =
    f(z/|X|) .
\]
Recall that, for a random walk on the graph with the starting point $\eta$, its Green function  is defined as
\[
    G(z)
    =
    \sum_{n=0}^\infty p^{(n)}(\eta,\eta)\, z^n
\]
where $p^{(n)} (\eta,\eta)$ is the probability for the random walk to come back to $\eta$ after exactly $n$ steps. The random walk is called simple if its transition probabilities are uniform on neighbouring vertices. 
Our \cref{thm:gfun} can be extended to calculate the Green function for the more general case of random walks with non-uniform transition probabilities. Beside being an important characteristic of the random walk, the Green function and its complexity is valuable in particular in the study of spectra of graphs (see for example \cite{Mohar1989}).

\medskip

A corollary to \cref{thm:main}
 is that the membership problem to $\Stab(G,X,\Stab(\eta))$ is decidable uniformly over all bounded automata group $G$ and uniformly over all eventually periodic rays $\eta$. Indeed, the membership to a fixed ET0L language is known to be decidable in space complexity $O(n)$ and time complexity $O(n^2)$ where $n$ is the length of
the input word (see Lemma 2.1 in \cite{leeuwen1976}); and it is possible to modify this procedure to make it uniform. 
However, uniform decidability of the subgroup membership problem for stabilisers of infinite rays follows from simpler arguments. 
In \cref{prop:solvable,prop:solvableb}, we address this decision problem where $\eta$ is a computable ray which is periodic or non-eventually periodic, respectively. 
The non-periodic case leads to a so-called \emph{promise problem}, and
in \cref{prop:unsolvable} we see that it cannot be extended to a decision procedure.
We thank the anonymous reviewer for their suggested proof sketch of \cref{prop:solvable} which we give in \cref{sec:bounded-automata-groups}.

\begin{restatable*}{proposition}{PropSolvability}\label{prop:solvable}
Suppose we are given as input
\begin{enumerate}
    \item 
    a finite symmetric set of automaton automorphisms $X$ of the tree $\Td = C^*$;
    \item a word $w\in X^*$ over the generating set $X$; and
    \item\label{it:req-eventually-periodic}
    two words $a,b\in C^*$ with $|b|\geq 1$.
\end{enumerate}
Then it is computationally decidable if $w\in\WP(\left\langle X\right\rangle,X,\Stab(ab^\omega))$.
That is, the subgroup membership problem is solvable uniformly over all bounded automata groups $G = \left\langle X\right\rangle$, and uniformly over all eventually periodic rays $\eta= ab^\omega$.

\end{restatable*}
We then treat the case of not eventually periodic computable rays and
show that membership to $\WP(\left\langle X \right\rangle,X,\Stab(\eta))$ is a so-called \emph{promise problem}, solvable uniformly over all finite sets $X$ of bounded automaton automorphisms and computable rays $\eta$, as long as we have a promise that the given ray is not eventually periodic.

This promise is necessary, as we see in \cref{prop:unsolvable}.

\begin{restatable*}{proposition}{PropSolvabilityb}\label{prop:solvableb}
Suppose we are given as input
\begin{enumerate}
    \item a finite symmetric set of bounded automaton automorphisms $X$ of the tree $\Td = C^*$;
    \item a word $w\in X^*$ over the generating set $X$; and
    \item\label{it:req-eventually-periodic2}
    a Turing machine $T$ which outputs an infinite ray $\eta\in C^\omega$.
\end{enumerate}
Then, given the promise that $T$ does not generate an eventually periodic ray, 
it is computationally decidable if $w\in \WP(\left\langle X \right\rangle,X,\Stab(\eta))$.
That is, membership to $\WP(\left\langle X\right\rangle,X,\Stab(\eta))$ is decidable as a promise problem uniformly for all bounded automata groups and uniformly for all non-eventually-periodic computable rays.
\end{restatable*}

It is interesting to note that one cannot remove the promise from \cref{prop:solvableb}, in fact, if one attempts to do so then the problem is no longer computable, as shown in the following Proposition.

\begin{restatable*}{proposition}{PropUnsolvability}\label{prop:unsolvable}
There is no Turing machine which can take as input
\begin{enumerate}
    \item a finite symmetric set of bounded automaton automorphisms $X$ of the tree $\Td=C^*$;
    \item a word $w\in X^*$ over the generating set $X$, and
    \item a Turing machine $T$ which outputs an infinite ray $\eta\in C^*$;
\end{enumerate}
then decide if $w\in \WP(\left\langle X \right\rangle,X,\Stab(\eta))$.
In particular, this means that the computation as described in \cref{prop:solvableb} cannot be generalised to a decision procedure, that is, the `promise' in \cref{prop:solvableb} is necessary.
\end{restatable*}

The paper is organised as follows. \Cref{sec:bounded-automata-groups,sec:et0l-language} are devoted to a general introduction to bounded automata groups and ET0L languages, but also serve to provide useful lemmas and definitions needed in the rest of the paper.
Section 2 also contains the proofs of \cref{prop:solvable,prop:solvableb,prop:unsolvable}. The proof of \cref{thm:main} is entirely contained in \cref{sec:main}, while \cref{sec:not-CF} gives criteria that can be used to show that the language of stabilisers is not context-free.
Lastly, \cref{sec:future} contains a list of open questions and topics for future research.

\section{Bounded Automata Groups}\label{sec:bounded-automata-groups}

As in the introduction, we write $\Td$ for the $d$-regular rooted tree. We identify the vertices with the words in $C^*$ where $C = \{c_1,c_2,...,c_d\}$, with the root labelled by the empty word $\varepsilon\in C^*$.
See \cref{fig:tree-vertex-labelling} for a depiction of this tree.

\begin{figure}[H]
	\centering
	\includegraphics{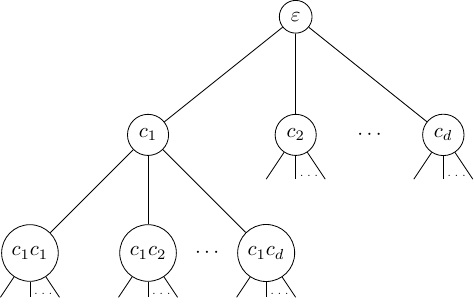}
	\caption{A labelling of the vertices of $\Td$.}
	\label{fig:tree-vertex-labelling}
\end{figure}

We write $\AutTd$ for the group of automorphisms of $\Td$. Every automorphism $\alpha \in \AutTd$ fixes the root and preserves the levels of the tree.
In fact, $\AutTd = \AutTd\wr\mathrm{Sym}(C)$ where $\mathrm{Sym}(C)$ is the symmetric group on the set $C$.
That is, each automorphism $\alpha\in \AutTd$ can be uniquely written in the form $\alpha = (\alpha'_1, \alpha'_2,...,\alpha'_d)\cdot s$ where each $\alpha'_i\in \AutTd$ is an automorphism of the subtree rooted at $c_i$ (which is isomorphic to $\Td$), and $s\in \mathrm{Sym}(C)$ is a permutation of the  subtrees rooted in the vertices of the first level.
The automorphism $\alpha'_i$, with $i=1,2,..., d$, is called the \emph{section of $\alpha$ at $c_i$}, and will be denoted $\alpha\at{c_i}$.
Then the \emph{section} $\alpha \at{v}$ at an arbitrary vertex $v=c_{i_1} c_{i_2} \cdots c_{i_m} \in C\star$ is defined recursively as
\[
	\alpha\at{v}
	=
	\left(\alpha\at{c_{i_1}}\right)\at{c_{i_2} c_{i_3} \cdots c_{i_m}}.
\]
That is, $\alpha\at v$ is the action that the element $\alpha$ has on the subtree rooted at $v$.

\begin{definition}\label{def:automata-automorphism}
An automorphism $\alpha \in \AutTd$ is an \emph{automaton automorphism} if there exists a finite set $A_\alpha\subset \AutTd$ such that $\alpha \at v \in A_\alpha$ for each $v\in C^*$.
The set of all automaton automorphisms forms a group $\mathrm{AAut}(\Td)$.
A group $G\leq\AutTd$ is called an \emph{automata group} if $G\leq \mathrm{AAut}(\Td)$.
\end{definition}

In the literature, the class of automaton automorphisms is usually introduced by first defining a computational model known as a finite-state automaton, for example, see Definition~1.3.1 in~\cite{Nekrashevych2005book}.
Our definition is equivalent, in particular, a finite set of states for an automaton representing an automaton automorphism $\alpha$ is given by the set $A_\alpha$ from \cref{def:automata-automorphism}.
Automata groups are often assumed to be self-similar (or state-closed). We do not make this assumption until Section 5 when it will be explicitly specified.

From the definition of an automaton automorphism, it is clear that it is determined by a finite amount of data. Therefore, any such automorphism can be encoded and provided as input to a Turing machine.
Moreover, it is computable to check if a given automaton is the identity, and composition of such automorphisms is also computable, see section 1.3.5 in \cite{Nekrashevych2005book} for more details.

We will be interested here in the subgroup membership problem in automata groups. As explained in the introduction, an important family of subgroups in a group of automorphisms of a regular rooted tree is formed by the stabilisers of the vertices of the tree and of the elements of the boundary of the tree (aka infinite rays emanating from the root). We noted in the introduction that the  subgroup membership problem in the stabiliser of a vertex of the tree is a regular language, and from now on we only concentrate on the problem of membership in the stabilisers of infinite rays.
At this point, we can show that for an automata group, the subgroup membership problem is decidable for stabilisers of eventually periodic rays.
We thank the anonymous reviewer for suggesting to us a sketch of the proof  of the following Proposition.

\PropSolvability

\begin{proof}
Suppose that we are given some finite set $X$ and segments $a,b\in C^*$.
Suppose also that we are given a word $w \in X^*$.
Then, we can compute a description of the corresponding element $\overline{w} \in \mathrm{AAut}(\Td)$.
In particular, suppose that $A_{\overline w}$ is the finite set of automorphisms as in \cref{def:automata-automorphism}.
Let $K = |A_{\overline w}|$.
As we have a description of the automaton automorphism $\overline w$, we can compute the vertex $v = (a\, b^{K+1}) \cdot \overline{w}$.
In the remainder of this proof, we show that $w \in \WP(\left\langle X \right\rangle, X, \Stab(\eta))$ if and only if $v = a\, b^{K+1}$.

Suppose that $v\neq a\, b^{K+1}$, then the action of $\overline{w}$ is non-trivial on some prefix of $\eta$ and thus is non-trivial on $\eta$. From this, we conclude that $w\notin \WP(\left\langle X \right\rangle, X, \Stab(\eta))$ as required.

Now suppose that $v = a\, b^{K+1}$, that is, that $\overline{w}$ has trivial action on the vertex $a\, b^{K+1}$.
By the pigeonhole principle, we then see that there must exist two distinct values $k_1,k_2\in \{0,1,...,K\}$ with $k_1> k_2$ such that
\[
    \beta\coloneqq \overline{w}\,\at\, (a \, b^{k_1}) = \overline{w}\,\at\, (a \, b^{k_2})
\]
where $b^{k_1-k_2} \cdot \beta = b^{k_1-k_2}$.
Thus, from the definition of the automaton automorphisms we see that
\[
    \eta\cdot\overline{w} = (ab^\omega)\cdot \overline{w}
    =
    (a b^k_2) (b^{k_1-k_2})^\omega=\eta.
\]
That is, $\overline{w}$ has trivial action on the ray $\eta$.
\end{proof}

From now on, we will be interested in  $\Td$ by \emph{bounded automaton automorphisms} defined as follows.

\begin{definition}\label{def:bounded-automata-automorphism}
We say that an automorphism $\alpha \in \AutTd$ is \emph{bounded}~\cite{sidki2000} if there exists some constant $N_\alpha$ such that
\[
    \#\{
        v\in C\star
    \mid
        \alpha\at{v} \neq 1
        \text{ and }
        |v|=k
    \} < N_\alpha
\]
for every positive integer $k$. The set of all bounded automaton automorphisms forms a group which we denote $\BTd < \mathrm{AAut}(\Td)$.
A group $G$ is called a \emph{bounded automata group} if  $G\leq \BTd$.
\end{definition}

In~\cite{sidki2000}, Sidki considered two classes of bounded automaton automorphisms, known as \emph{finitary} and \emph{directed} automaton automorphisms, and showed that they form a generating set for the group of bounded automaton automorphisms $\BTd$, see \cref{prop:ba_gset}.

\begin{definition}
An automorphism $\phi \in \AutTd$ is \emph{finitary} if there exists a constant $N_\phi \in \N$ such that $\phi\at{v} = 1$ for each $v\in C\star$ with $|v|\geq N_\phi$.
The smallest constant for which this holds is the \emph{depth} of the automorphism $\phi$, denoted as $\mathrm{depth}(\phi)$.
\end{definition}

Finitary automorphisms form a subgroup of $\BTd$ which we denote as $\FinTd$.
Examples of finitary automorphisms are given in \cref{fig:finitary-examples}

\begin{figure}[h!t]
	\centering
	\begin{minipage}[t]{.3\linewidth}
		\centering
		\includegraphics{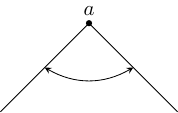}
	\end{minipage}
	~
	\begin{minipage}[t]{.3\linewidth}
		\centering
		\includegraphics{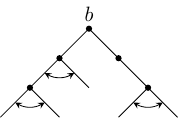}
	\end{minipage}
	\caption{Examples of finitary automorphisms $a,b\in\FinTd$.}
	\label{fig:finitary-examples}
\end{figure}

For any automorphism $\phi \in \AutTd$, finitary or not, we will also need the notion of its \lq\lq directional depth\rq\rq with respect to a vertex of the tree or with respect to an infinite ray, that we now introduce.

\begin{definition}\label{def:compdepth}
The \emph{directional depth} of  an automorphism with respect to a word, finite or infinite, in the alphabet $C$ is given by the function
$\CompDepth\colon (C^*\cup C^\omega) \times \AutTd \to \mathbb N \cup \{\infty\}$ defined as
\[
    \CompDepth(\zeta,x)
    =
     \begin{cases}
    \min\{ 
        |v|
    \mid
        v \text{ is a prefix of }\zeta\text{ with }x\at v =1
    \}
    \\
     \infty \ \text{ if } x\at{v}\neq 1 \text{ for each prefix } v \text{ of } \zeta .
     \end{cases}
\]

\end{definition}

Let us now turn to directed automorphisms. 

\begin{definition} \label{def:directed}
A bounded automaton automorphism $\delta \in \AutTd$ is \emph{directed} if there exists a unique infinite word $c_{i_1}c_{i_2}\ldots c_{i_m} \ldots$ such that $\delta\at c_{i_1}c_{i_2}\ldots c_{i_m} \neq 1$ for all $m$. Such word is called \emph{spine} and denote by $\mathrm{spine}(\delta)$.
\end{definition}

We denote the set of all directed automaton automorphisms as $\DirTd$.
See \cref{fig:directed-examples} for some examples (in these examples, $a$ and $b$ are as in \cref{fig:finitary-examples}).
\begin{figure}[h!t]
	\centering
	\begin{minipage}[t]{.3\linewidth}
		\centering
		\includegraphics{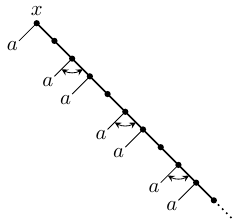}
	\end{minipage}
	~
	\begin{minipage}[t]{.3\linewidth}
		\centering
		\includegraphics{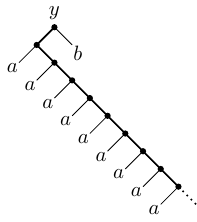}
	\end{minipage}
	~
	\begin{minipage}[t]{.3\linewidth}
		\centering
		\includegraphics{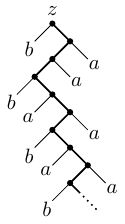}
	\end{minipage}
	\caption{Examples of directed automorphisms $x,y,z \in \mathrm{Dir}(\mathcal{T}_2)$.}
	\label{fig:directed-examples}
\end{figure}

The composition of two elements in $\DirTd$ might not be in $\DirTd$. However, the resulting automorphism has at most two infinite rays with the same behaviour of the spines of the directed automorphisms. Indeed, if $\delta, \delta' \in \DirTd$, then either $\mathrm{spine}(\delta)\cdot\delta=\mathrm{spine}(\delta')$ and $\delta\delta'$ is still a directed automorphism with $\mathrm{spine}(\delta\delta')=\mathrm{spine}(\delta)$, or $\delta\delta'$ has two infinite rays $\mathrm{spine}(\delta)$ and $\mathrm{spine}(\delta')\cdot \delta^{-1}$ such that the restrictions on each prefix is non-trivial (note that $\delta^{-1}$ is acting as a finitary automorphism on $\mathrm{\delta'}$). The same holds more generally for a product of directed automorphisms.

\begin{proposition}[Proposition 16 in \cite{sidki2000}]\label{prop:ba_gset}
	The group $\mathcal{B}(\Td)$ of bounded automaton automorphisms is generated by $\FinTd$ together with $\DirTd$.
\end{proposition}

The proof of Proposition~16 in \cite{sidki2000} is constructive, that is, given a bounded automaton automorphism $\alpha$, it is computable to find a finite decomposition
    $
        \alpha = s_1 s_2 \cdots s_k
    $
    where each $s_i\in \FinTd \cup \DirTd$.
This follows since it is composition and equality is computable in the set of automaton automorphisms.
  In particular, given a bounded automaton automorphism, $\alpha$, one can nondeterministically choose such a decomposition $s_1 s_2 \cdots s_k$, where each $s_i\in \FinTd\cup\DirTd$, then verify that it presents the same automorphism as $\alpha$.

\begin{proposition}[Lemma~3 on~p.~87 of~\cite{Bishop2019}]
\label{lemma:spine is eventualy periodic}
	The spine, $\mathrm{spine}(\delta) \in C^\omega$, of a directed automaton automorphism, $\delta \in \DirTd$, is eventually periodic, i.e., there are words $u = u_1 u_2 \cdots u_s \in C\star$ and $v = v_1 v_2 \cdots v_t \in C\star$ with $v \neq \varepsilon$, called the \emph{initial} and \emph{periodic} segment respectively, for which $\mathrm{spine}(\delta) = u v^\omega$ and
	\[
		\delta
		\at{u v^k v_1 v_2 \cdots v_j}
		=
		\delta
		\at{u v_1 v_2 \cdots v_j}
	\]
	for each $k,j \in \N$ with $0\leq j \leq t$.
\end{proposition}

This generalises to any element of $\BTd$. Since the product of a directed automorphism and a finitary automorphism is directed and by the discussion right after \cref{def:directed} on the product of directed automorphism, by \cref{prop:ba_gset} is clear that a bounded automaton automorphism has finitely many infinite rays such that on each prefix the restriction is non-trivial. Moreover, the action of an element acting like a finitary automorphism on an infinite word only changes a finite prefix, and by the same discussion as above, we have that the infinite rays of a bounded automaton automorphism are all eventually periodic.

We are now ready to prove the following lemma which is used to simplify the proof of \cref{thm:main}.

\begin{lemma}\label{lem:fg.ss.ba_overgroup}
Let $G < \BTd$ be a finitely generated bounded automata group. There exists a bounded automata group $H$ with a finite symmetric generating set $S \subset \FinTd\cup \DirTd$ such that $G$ is a subgroup of $H$.
Moreover, such a generating set $S$ is effectively constructable from a finite set of bounded automaton automorphisms $X$ where $G = \left\langle X \right\rangle$.
\end{lemma}

\begin{proof}
Let $X$ be a finite generating set for the group $G$.
From \cref{prop:ba_gset}, we see that for each $x\in G$,  in particular, for  each $x\in X$, there is a word $x = w_{x,1} w_{x,2} \cdots w_{x,k(x)}$ with each $w_{x,i} \in \FinTd\cup \DirTd$.
Define a symmetric generating set
\[
    S
    =
    \{ 
        w_{x,j}
        ,
        (w_{x,j})^{-1}
    \mid    x\in X, j\in \{1,2,...,k(x)\}
    \}.
\]
From the definition of finitary and directed automaton automorphisms, we see that $S\subset \FinTd\cup\DirTd$, as desired.
Moreover, the group generated by $S$ contains $G$ as a subgroup.
It follows from the observation  after \cref{prop:ba_gset} that the finite set $S$ is computable from $X$.
\end{proof}

In \cref{prop:solvable}, we showed that the subgroup membership problem for the stabiliser subgroup of an eventually periodic ray is computable for arbitrary finitely generating sets of automaton automorphisms.
In the case of bounded automata automorphisms, we are able to show the following companion property for non-eventually-periodic computable rays.
However, it is a promise problem, moreover, we show in \cref{prop:unsolvable} below that it cannot be turned into a proper decision problem.

\PropSolvabilityb

\begin{proof}
We begin by computing an automaton automorphism for the action of $\overline{w}$.\ 

Since the infinite rays of a bounded automaton automorphism is eventually periodic (see \cref{lemma:spine is eventualy periodic} and the discussion right after it), we then see that the ray $\eta$, as described by $T$, must eventually leave this finite set of rays.
Thus, we see that $\overline{w}$ only performs an action on a finite prefix of $\eta$.
Moreover, the length of this prefix is computable from the description of the automaton for $\overline{w}$ and prefixes of $\eta$.
From these observations, it follows immediately that $w\in \WP(\left\langle X \right\rangle, X, \Stab(\eta))$ is decidable.
\end{proof}

\PropUnsolvability
\begin{proof}
We begin by introducing a decision problem which we call \textit{Periodicity}:
\begin{description}
  \item[Input:] A Turing-machine-based description of an infinite ray $\eta \in \{0,1\}^{\omega}$.
  \item[Question:] Is $\eta = 1^{\omega}$?
\end{description}
Periodicity is essentially a reformulation of the halting problem. For completeness, let us outline the reduction.  
Given a Turing machine $T$, construct a machine $P$ that simulates $T$ and, at each step of the simulation, outputs the letter $1$. If $T$ ever halts then $P$ switches and outputs only $0$’s from that point onward, thus producing the sequence $1^k 0^{\omega}$ for some $k\in\N$.  
Hence, $P$ outputs the infinite sequence $1^{\omega}$ if and only if $T$ does not halt. Therefore, if the problem \textit{Periodicity} were decidable, then the halting problem would also be decidable, a contradiction.

We now reduce \textit{Periodicity} to the problem in the statement. Consider the infinite dihedral  group $\mathcal{D}$ (see \cref{fig:dihedral-automaton}), a well-known bounded automaton group generated by $a$ and $b$. It is straightforward to check that for an infinite ray $\eta$ one has $\eta \cdot b = \eta$ if and only if $\eta = 1^{\omega}$.
Suppose by contradiction that our problem is decidable. Then there exists a Turing machine $M$ that, given a bounded automaton group $G$, a Turing machine $T_{\eta}$ producing the infinite ray $\eta$, and a word $w\in X^*$ over the generating set $X$ of $G$, decides whether $w \in \WP(G,X,\Stab(\eta))$. Applied to the input $(\mathcal{D},T_{\eta},b)$, the machine $M$ would decide whether
$b \in \WP(\mathcal{D},\{a,b\},\Stab(\eta))$, that is, whether $\eta \cdot b = \eta$, which is equivalent to deciding whether $\eta = 1^{\omega}$. Thus, $M$ would solve \textit{Periodicity}, which we have shown to be undecidable. This contradiction shows that such a machine $M$ cannot exist.  
\end{proof}

\section{ET0L Languages}\label{sec:et0l-language}

In this section, we define and provide a background on the family of \emph{Extended Tabled 0-interaction Lindenmayer} (\emph{ET0L}) languages (see \cref{def:et0l grammar}) which was introduced and studied by Rozenberg~\cite{Rozenberg1973}.
We begin by defining the ET0L languages in terms of a class of formal grammars.
We conclude this section by studying a particular subclass of ET0L language in \cref{sec:derivation-tree} and show in \cref{sec:gfun} that ET0L languages from this class have generating functions which we can specify using equations of a particular form.

Below, we give a definition of ET0L languages which is due to Asveld~\cite{asveld1977}.
In particular, the definition we use in this paper is what Asveld refers to as a $(\mathrm{REG},\mathrm{REG})$ITER grammar (cf.~the definitions on pp.~253-4 of~\cite{asveld1977}).
The proof that this is equivalent to the definition given by Rozenberg in \cite{Rozenberg1973} follows from Theorem~2.1 in \cite{asveld1977}, and Theorems~2 and~3 in \cite{Nielson1975}.
(Note that RC-Part ET0L in \cite{Nielson1975} has the same definition as $(\mathrm{REG})$ITER in \cite{asveld1977}).

An \emph{ET0L grammar} is a type of replacement system which has both a terminal alphabet $\Sigma$ and a disjoint nonterminal alphabet $V$.
In particular, our grammar begins with an initial symbol $S\in V$.
We then perform a sequence of allowable replacements to this symbol until we have a word which consists of only letters in $\Sigma$.
Such a word is then said to be generated by the grammar.
Allowable replacements are given by \emph{tables}, defined as follows.

\begin{definition}\label{def:tables}
    A \emph{table} is a function of the form $\tau\colon \Sigma\cup V\to \Reg(\Sigma\cup V)$ where $\Reg(\Sigma\cup V)$ denotes the family of regular languages over the alphabet $\Sigma\cup V$ and $\tau(\sigma)=\{\sigma\}$ for each $\sigma\in \Sigma$.
\end{definition}

Since the elements of $\Sigma$ are fixed, we do not specify them when we explicitly provide a table. Suppose that $\tau \colon \Sigma\cup V \to \Reg(\Sigma \cup V)$ is a table as defined above.
Then we write $w\to^\tau w'$ for each word $w = w_1 w_2 \cdots w_m \in (\Sigma\cup V)\star$ and each word $w' = w_1' w_2'\cdots w_m'$ where each $w_i'$ belongs to the regular language $ \tau(w_i)$. 
For tables $\tau_1$, $\tau_2$, \ldots, $\tau_k$, we write $w\to^{\tau_1 \tau_2 \cdots \tau_k} w'$ if there are words $w_1,w_2,\ldots,w_{k+1} \in (\Sigma\cup V)\star$ with $w_1=w$, $w_{k+1} = w'$ and $w_i \to^{\tau_i} w_{i+1}$ for each $i$.

For example, let $\Sigma = \{a,b\}$ and $V = \{S,A,B\}$, then
\begin{equation}\label{eq:alternative-tables-example}
	\alpha \colon
	\left\{
	\begin{aligned}
		S &\mapsto \{SS, S, AB\}\\
		A &\mapsto \{A\}\\
		B &\mapsto \{B\}
	\end{aligned}
	\right.
	\qquad
	\beta \colon
	\left\{
	\begin{aligned}
		S &\mapsto \{S\}\\
		A &\mapsto \{aA\}\\
		B &\mapsto \{bB\}
	\end{aligned}
	\right.
	\qquad
	\gamma \colon
	\left\{
	\begin{aligned}
		S &\mapsto \{S\} \\
		A &\mapsto \{\varepsilon\} \\
		B &\mapsto \{\varepsilon\}
	\end{aligned}
	\right.
\end{equation}
are tables.
We see that $w\to^\alpha w'$ where $w=SSSS$ and $w' = SABSSAB$.

We can now define ET0L grammars, as follows.

\begin{definition}\label{def:et0l grammar}
    An \emph{ET0L grammar} is a 5-tuple $E = (\Sigma, V, T, \mathcal{R}, S)$, where
	\begin{enumerate}
		\item $\Sigma$ is an alphabet of \emph{terminals};
		\item $V$ is an alphabet of \emph{nonterminals};
		\item\label{def:et0l grammar/tables} $T = \{\tau_1, \tau_2 ,\dots, \tau_k\}$ is a finite set of \emph{tables},
		\item $\mathcal{R} \subseteq T\star$ is a regular language called the \emph{rational control}; and
		\item $S \in V$ is the \emph{start symbol}.
	\end{enumerate}
    We then say that
	\[
		L(E)
		=
		\left\{
			w \in \Sigma\star
		\mid 
			S \to^{v} w
			\text{ for some }
			v \in \mathcal{R}
		\right\}
	\]
    is the \emph{ET0L language} generated by the grammar $E$
\end{definition}

For example, let $\alpha$, $\beta$ and $\gamma$ be as in (\ref{eq:alternative-tables-example}), then the language that is produced by the grammar with rational control $\mathcal{R} = \alpha\star \beta\star \gamma$ is $\{( a^n b^n )^m \mid n,m \in \mathbb{N} \}$.
It can then be shown, using the \textit{pumping lemma} (see~\cite[Theorem~2.34]{S2013}), that this language is not context-free.
It is known that every context-free language is also ET0L (see the diagram in T28 on p.~241 of \cite{Rozenberg1986}).

We now introduce some additional notation which will be used in the proof of \cref{thm:main}.
We prefer this notation as it matches the way in which we apply tables from left to right.

\begin{notation}\label{notation:table-right-action}
    Suppose that $E = (\Sigma, V, T, \mathcal R, S)$ is an ET0L language, then for each word $w\in (\Sigma\cup V)^*$ and each sequence of tables $t\in T^*$, we write
    \[
        w \cdot t
        =
        \{
            u \in (\Sigma\cup V)^*
        \mid
            w\to^t u
        \}
    \]
    for the set of all words which can be obtained from $w$ by applying $t$.
\end{notation}

\subsection{Unambiguous Limiting Grammars}\label{sec:derivation-tree}

Each of the ET0L grammars that we construct in \cref{thm:main} has a particular form for which one can compute a description of their generating function (see \cref{thm:gfun}).
In this subsection, we give a description of this class of grammars.
Later in this section, we prove some closure properties, and study the combinatorial complexity of this class.
We begin by describing we mean for a grammar to be \emph{unambiguous} as follows.

Similarly to context-free languages, we can define derivation trees for ET0L grammars.
However, for the derivation tree from an ET0L language, we label each level of the tree to denote the table which is being applied.
For example, consider the language of partitions given as
\[
    L
    =
    \{
        a^{n_1} b a^{n_2} b \cdots b a^{n_k} b
    \mid
        k \geq 1\text{ and }
        n_1 \geq n_2 \geq \cdots \geq n_k \geq 1
    \}.
\]
It was shown in \cite{ciobanu2018} that this language is ET0L, in particular, it is generated by an ET0L language with nonterminals $S$ and $A$, and tables
\[
    \alpha\colon
    \left\{
    \begin{aligned}
        S &\mapsto aAb S\\
        A &\mapsto A
    \end{aligned}
    \right.
    \qquad
    \beta\colon
    \left\{
    \begin{aligned}
        S &\mapsto S\\
        A &\mapsto aA
    \end{aligned}
    \right.
    \qquad
    \text{and}
    \qquad
    \gamma\colon
    \left\{
    \begin{aligned}
        S &\mapsto \varepsilon\\
        A &\mapsto \varepsilon
    \end{aligned}
    \right.
\]
with rational control $\mathcal{R} = \{\alpha, \beta, \gamma\}\star$.
Notice that the word $a^2 b a b a b$ belongs to the language $L$. In particular, this word has a derivation tree labelled by $\alpha\beta\alpha\alpha\gamma$ as given in \cref{fig:derivation-tree}.
We obtain the word $a^2 b abab$ from the tree given in \cref{fig:derivation-tree} by reading off the leaves from left to right.

\begin{figure}[ht!]
\centering

\begin{tikzpicture}

\node (root) at (0,0) {$S$};

\node (l0) at ($(root) - (3,0)$) {$\alpha$};

\draw[dotted] (l0) -- (root);


\node[draw] (l1n1) at ($(root)-(1.5,1)$)  {$a$};
\node (l1n2) at ($(root)-(0.5,1)$)  {$A$};
\node[draw] (l1n3) at ($(root)-(-0.5,1)$) {$b$};
\node (l1n4) at ($(root)-(-1.5,1)$) {$S$};

\draw[-] (root.south) -- (l1n1.north);
\draw[-] (root.south) -- (l1n2.north);
\draw[-] (root.south) -- (l1n3.north);
\draw[-] (root.south) -- (l1n4.north);

\node (l1) at ($(l0)-(0,1)$) {$\beta$};
\draw[dotted] (l1) -- (l1n1) -- (l1n2) -- (l1n3) -- (l1n4);


\node[draw] (l2n1) at ($(l1n2)-(+0.5,1)$)  {$a$};
\node (l2n2) at ($(l1n2)-(-0.5,1)$)  {$A$};
\node (l2n3) at ($(l1n4)-(0,1)$) {$S$};

\draw[-] (l1n2.south) -- (l2n1.north);
\draw[-] (l1n2.south) -- (l2n2.north);
\draw[-] (l1n4.south) -- (l2n3.north);

\node (l2) at ($(l1)-(0,1)$) {$\alpha$};
\draw[dotted] (l2) -- (l2n1) -- (l2n2) -- (l2n3);


\node (l3n1) at ($(l2n2)-(0,1)$)  {$A$};

\node[draw] (l3n2) at ($(l2n3)-(1.5-1,1)$)  {$a$};
\node (l3n3) at ($(l2n3)-(0.5-1,1)$)  {$A$};
\node[draw] (l3n4) at ($(l2n3)-(-0.5-1,1)$) {$b$};
\node (l3n5) at ($(l2n3)-(-1.5-1,1)$) {$S$};

\draw[-] (l2n2.south) -- (l3n1.north);
\draw[-] (l2n3.south) -- (l3n2.north);
\draw[-] (l2n3.south) -- (l3n3.north);
\draw[-] (l2n3.south) -- (l3n4.north);
\draw[-] (l2n3.south) -- (l3n5.north);

\node (l3) at ($(l2)-(0,1)$) {$\alpha$};
\draw[dotted] (l3) -- (l3n1) -- (l3n2) -- (l3n3) -- (l3n4) -- (l3n5);


\node (l4n1) at ($(l3n1)-(0,1)$)  {$A$};
\node (l4n2) at ($(l3n3)-(0,1)$)  {$A$};

\node[draw] (l4n3) at ($(l3n5)-(1.5-1,1)$)  {$a$};
\node (l4n4) at ($(l3n5)-(0.5-1,1)$)  {$A$};
\node[draw] (l4n5) at ($(l3n5)-(-0.5-1,1)$) {$b$};
\node (l4n6) at ($(l3n5)-(-1.5-1,1)$) {$S$};

\draw[-] (l3n1.south) -- (l4n1.north);
\draw[-] (l3n3.south) -- (l4n2.north);

\draw[-] (l3n5.south) -- (l4n3.north);
\draw[-] (l3n5.south) -- (l4n4.north);
\draw[-] (l3n5.south) -- (l4n5.north);
\draw[-] (l3n5.south) -- (l4n6.north);

\node (l4) at ($(l3)-(0,1)$) {$\gamma$};
\draw[dotted] (l4) -- (l4n1) -- (l4n2) -- (l4n3) -- (l4n4) -- (l4n5) -- (l4n6);


\node (l5n1) at ($(l4n1)-(0,1)$)  {$\varepsilon$};
\node (l5n2) at ($(l4n2)-(0,1)$)  {$\varepsilon$};
\node (l5n3) at ($(l4n4)-(0,1)$) {$\varepsilon$};
\node (l5n4) at ($(l4n6)-(0,1)$) {$\varepsilon$};

\draw[-] (l4n1.south) -- (l5n1.north);
\draw[-] (l4n2.south) -- (l5n2.north);
\draw[-] (l4n4.south) -- (l5n3.north);
\draw[-] (l4n6.south) -- (l5n4.north);
\end{tikzpicture}
\caption{Derivation tree for $a^2 b a b a b$ labelled by $\alpha\beta\alpha\alpha\gamma$.}\label{fig:derivation-tree}
\end{figure}
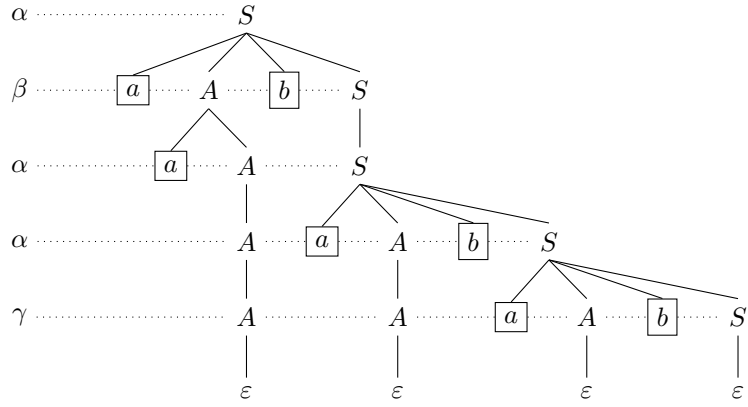

We then say that a derivation tree as in \cref{fig:derivation-tree} is a derivation tree labelled by $\alpha\beta\alpha\alpha\gamma$ and that for each word $w$ with $S \to^{\alpha\beta\alpha\alpha\gamma} w$, there is a derivation tree with the same labelling.
Moreover, by considering other first 4 levels of \cref{fig:derivation-tree}, we see that the word $aaAaAbS$ has a derivation tree labelled by $\alpha\beta\alpha\alpha$. That is, we allow our derivation trees to have nonterminals in their leaves.
We now define unambiguous ET0L languages as follows.

\begin{definition}\label{def:unambiguous-et0l}
    Let $E$ be an ET0L grammar, then we say that $E$ is \emph{unambiguous with respect to its rational control} (or simply \emph{unambiguous}) if for every $r\in \mathcal R$, and every word $w \in (\Sigma\cup V)^*$, there is at most one derivation tree for $w$ labelled by $r$.\footnote{Here we consider all words over terminals and nonterminals, not just words only in terminal letters.}
\end{definition}

We now define the class of ET0L grammars, as in the proof of \cref{thm:main}, as follows.

\begin{definition}\label{def:limiting-et0l}
    Let $E = (\Sigma, V, T,\mathcal{R}, S)$ be an ET0L grammar as in \cref{def:et0l grammar}, then we say that $E$ is \emph{limiting} if
    \begin{enumerate}
        \item\label{def:limiting-et0l/1} $T$ contains 3 tables, i.e., $T = \{\alpha,\beta,\gamma\}$;
        \item\label{def:limiting-et0l/2} the rational control is given by $\mathcal{R} = \alpha\beta\star\gamma$;
        \item\label{def:limiting-et0l/beta} we have $\varepsilon \notin \beta(v)$ for each $v \in V$;
        \item\label{def:limiting-et0l/gamma} we have $\gamma(v) \neq \emptyset$ for each $v \in V$;
        \item\label{def:limiting-et0l/limit}
        (limiting) if $S\to^{\rho_n} w$ for some $w \in (\Sigma\cup V)^*$ where $\rho_n = \alpha\beta^n\gamma$, then there exists some $K \geq n$ such that $S\to^{\rho_k} w$ for each $k \geq K$.
    \end{enumerate}
    Requirements (\ref{def:limiting-et0l/beta}) and (\ref{def:limiting-et0l/gamma}) are technical requirements which allow us to compute the generating function.
\end{definition}

We then say that an ET0L grammar is \emph{unambiguous limiting} if it satisfies the properties in \cref{def:unambiguous-et0l,def:limiting-et0l}.
In the following subsection we show that the class of unambiguous limiting grammars is closed under mappings by injective string transducers.
This result is then used to prove \cref{prop:closure-under-gset} which is used to simplify the proof of our main theorem (i.e.~\cref{thm:main}).

\subsection{Closure under mapping by string transducer}

We now show that an unambiguous limiting ET0L language is closed under mapping by a \emph{string transducer}, also known as a \emph{deterministic finite-state transducer}, or a \emph{deterministic generalised sequential machine (deterministic gsm)}.
We begin with the following definition.

\begin{definition}\label{def:det-fst}
	A (deterministic) \emph{string transducer} is a tuple
	$M = (\Gamma,\penalty51 \Sigma, Q,\penalty51 A,\penalty51 q_0, \penalty51\delta)$ where
	\begin{itemize}
    \item $\Gamma$ and $\Sigma$ are the \emph{input} and \emph{output alphabets}, respectively;
    \item $Q$ is a finite set of \emph{states};
    \item $A\subseteq Q$ is a finite set of \emph{accepting states};
    \item $q_0 \in Q$ is the \emph{initial state}; and
    \item $\delta\colon \Gamma \times Q \to \Sigma^*\times Q$ is a \emph{transition function}.
	\end{itemize}
	Given a language $L\subseteq \Gamma^*$, we may then define the language $M(L)\subseteq \Sigma^*$ as
	\[
		M(L)
		=
		\left\{
		u_1 u_2 \cdots u_k\in \Sigma^*
		\ \middle|\
		\begin{aligned}
			\text{there exists some word }
			w = w_1 w_2 \cdots w_k \in L\subseteq \Gamma^*
			\\\text{ such that }
			\delta(w_i, q_{i-1}) = (u_i, q_i)
			\text{ for each }i \in \{1,2,\ldots,k\} \\
			\text{ where }
			q_0\text{ is the initial state, and }
			q_1,q_2,\ldots,q_k\in Q
			\text{ with }
			q_k\in A
		\end{aligned}
		\right\}.
	\]
	We then say that $M(L)$ is the image of $L$ under mapping by the string transducer $M$.
\end{definition}

	Let $M = (\Gamma,\Sigma,Q,A,q_0,\delta)$ be a string transducer.
	Then, for each pair of states $q,q'\in Q$, and words $w = w_1 w_2 \cdots w_k\in \Gamma^*$ and $w'\in \Sigma^*$, we write $q\to^{(w,w')}q'$ if there is a path from state $q$ to $q'$ which rewrites the word $w$ to $w'$; that is, if there is a sequence of states $q_1,q_2,\ldots,q_{k+1}\in Q$ such that
	\begin{itemize}
		\item $q = q_1$ and $q' = q_{k+1}$; and
		\item $\delta(w_i, q_i) = (u_i, q_{i+1})$ for each $i\in \{1,2,\ldots,k\}$ where $w' = u_1 u_2 \cdots u_k$.
	\end{itemize}
We then see that the language $M(L)$ can be written as
	\[
		M(L)
		=
		\{
		w'\in \Sigma^*
		\mid
		q_0 \to^{(w,w')} q
		\text{ where }
		w\in L\text{ and }q\in A
		\}
	\]
	for each $L\subseteq \Gamma^*$.

\begin{definition}
We say that a string transducer $M = (\Gamma,\Sigma,Q,A,q_0,\delta)$ is \emph{injective} if for each $w'\in  \Sigma^*$, there is at most one word $w\in \Gamma^*$ such that $q_0\to^{(w,w')} q$ with $q\in A$.
    That is, if we view $M$ as a partial map from $\Gamma^*$ to $\Sigma^*$, then it is injective in the usual sense.
\end{definition}

Our main objective in this subsection is to prove the following proposition.

\begin{proposition}\label{prop:string-transducer}
Suppose that we are given an unambiguous limiting ET0L grammar for some language $L\subseteq \Sigma^*$, and an injective string transducer $M = (\Sigma,\Gamma,Q,A,q_0,\delta)$.
Then, it is computable to find an unambiguous limiting ET0L grammar for the language $L' = M(L)$.
\end{proposition}

\begin{proof}
Let $E = (\Sigma,V,T,\mathcal R, \mathcal S)$ be an unambiguous limiting ET0L grammar for the language $L$, and let $M = (\Sigma,\Gamma,Q,A,q_0,\delta)$ be a string transducer.
We are constructing an ET0L grammar $E' = (\Sigma,V',\mathcal R', \mathcal S')$ which recognises the language $L' = M(L)$. The grammar $E'$ is obtained by annotating the nonterminals in $E$ with paths in the string transducer $M$.
We then modify the grammar $E'$ and show that it is unambiguous limiting.
We begin our construction by listing out the nonterminals of $E'$ as follows.

\medskip
\noindent\underline{Nonterminals}:

\smallskip
\noindent
We begin by introducing the initial nonterminals of $E'$ which we write as $\mathcal S'\in V$.
For each nonterminal $v\in V$, and each pair of states $q,q'\in Q$, we introduce a nonterminal $\Sigma_{v,q,q'}\in V'$ which corresponds to a word produced by $E$, from the nonterminal $v$, which is read along a path from state $q$ to $q'$ in $M$.
Similarly, for each terminal $x\in \Sigma$, and each $q,q'\in Q$, we introduce a nonterminal $\Sigma_{x,q,q'}\in V'$ which corresponds to a letter, produced by $E$, which is read along an edge between states $q$ and $q'$ in the string transducer $M$.
To simplify our proof, we allow nonterminals of the form $\Sigma_{x,q,q'}$ where there is no edge from $q$ to $q'$ labelled by $a$.
We only verify that a valid path is being represented at the end of the construction.

We now describe the tables of the grammar $E'$, beginning with an initialisation table as follows.

\medskip
\noindent \underline{Initialisation table}: $\tau_\mathrm{init}$.

\smallskip
\noindent
The table $\tau_\mathrm{init}$ decides on an accepting path in the automaton $M$ and is defined as follows:
\[
    \tau_\mathrm{init}(\mathcal S') = \{ \Sigma_{\mathcal S, q_0,q'} \mid q'\in A \},
\]
and $\tau_\mathrm{init}(v)=v$ for all other nonterminals $v$.
Observe that $\tau_\mathrm{init}$ is a table as each of its replacements are finite sets which are all regular languages.

Such a table `guesses' that the grammar will produce a word which corresponds to a path from $q_0$ to $q'\in A$.
We now describe a modification of the table $\alpha\in T$, and  we modify the tables $\beta$ and $\gamma$ in precisely the same manner.

\medskip
\noindent
\underline{Modifying tables $\alpha,\beta,\gamma\in T$}.

\smallskip
\noindent
We introduce a table $\alpha'\in T'$ such that
\begin{itemize}
    \item
    for each $v\in V$, and $q,q'\in Q$, we have
    \[
        \Sigma_{s_1, q,q_1}
        \Sigma_{s_2, q_1,q_2}
        \Sigma_{s_3, q_2,q_3}
        \cdots
        \Sigma_{s_{k+1}, q_k,q'}
        \in \alpha'(\Sigma_{v,q,q'})
    \]
    for every $q_1,q_2,...,q_k\in Q$, if and only if $s_1 s_2\cdots s_k \in \alpha(v)$.
    Note that we only require that the states of adjacent letters match.
    Thus, it can be seen that the language $\alpha'(\Sigma_{v,q,q'})$ is regular; and
    \item 
    for all other nonterminals $x\in \Sigma'$, we have $\alpha'(x)= x$.
\end{itemize}
We perform the same modification to tables $\beta,\gamma \in T$ to obtain tables $\beta',\gamma'\in T'$, respectively.

\medskip
\noindent
\underline{Observation I}:

\smallskip
\noindent
From the definition of the tables $\tau_\mathrm{init},\alpha',\beta',\gamma'\in T'$, we  see that for every choice of states $q,\penalty51 q',\penalty51q_1,\penalty51q_2,\penalty51...,\penalty51q_k\in Q$, symbols $s_1,s_2,...,s_{k+1}\in \Sigma\cup V$, and $n\in \N$,
\[
        \Sigma_{s_1, q,q_1}
        \Sigma_{s_2, q_1,q_2}
        \Sigma_{s_3, q_2,q_3}
        \cdots
        \Sigma_{s_{k+1}, q_k,q'}
        \in
        \mathcal S' \cdot \bigl(\tau_\mathrm{init}\alpha'(\beta')^n \gamma'\bigr)
\]
if and only if $q=q_0$, $q'\in A$ and
\[
    s_1 s_2 s_3 \cdots s_{k+1} \in \mathcal S \cdot \bigl(\alpha \beta^n \gamma\bigr).
\]
All that now remains is to introduce one additional table to apply the action of the string transducer $M$.

\medskip
\noindent
\underline{Final table}: $\tau_\mathrm{final}$.

\smallskip
\noindent
The additional table $\tau_\mathrm{final}$ is defined as follows.
For each $\Sigma_{a,q,q'}$ with $a\in \Sigma$, we define
\[
    \tau_\mathrm{final}(\Sigma_{a,q,q'})
    =
    \begin{cases}
        w & q\to^{(a,w)} q'\\
        \Sigma_{a,q,q'} & \text{otherwise}.
    \end{cases}
\]
For all other nonterminals $v\in V$, we define $\tau_\mathrm{finish}(v)=v$.

\medskip
\noindent
\underline{Observation II}.

\smallskip
\noindent
For each $w\in \Gamma^*$, we have
\[
    w \in \mathcal S' \cdot \bigl( \tau_\mathrm{init} \alpha' (\beta')^n \gamma' \tau_\mathrm{finish} \bigr)
\]
if and only if there is some $w' \in \Sigma^*$ such that
\[
    w' \in \bigl( \alpha \beta^n \gamma \bigr)
\]
with $q_0 \to^{(w',w)} q$ for some $q\in A$.

\medskip
\noindent
\underline{Modifying the grammar $E'$}.

\smallskip
\noindent
We construct the tables $\alpha'' = \tau_\mathrm{init}\alpha'$, $\beta'' = \beta'$ and $\gamma'' = \gamma' \tau_\mathrm{finish}$.
Then, with the rational control $\mathcal R = \alpha'' (\beta'')^*\gamma''$, we have an ET0L grammar for the language $L' = M(L)$.
Moreover, the grammar $E'$ is unambiguous because $E$ is unambiguous, and it is limiting by 
 the construction of the tables.
\end{proof}

We use \cref{prop:string-transducer} to prove \cref{prop:closure-under-gset} which is then used in the proof of our main theorem.
In order to prove this proposition, we first require some technical lemmas given as follows.
Recall that a subset $W\subset \Gamma^+$ is an \emph{antichain} with respect to prefix order if for each choice of words $u,v\in W$, the word $u$ is not a proper prefix of $v$.

\begin{lemma}\label{lem:f_is_detfsm}
	Let $W\subset \Gamma^+$ be a finite antichain with respect to prefix order.
	For each word $w\in W$, we fix a word $x_w\in \Sigma^*$.
	Define a map $f\colon \mathcal{P}(\Gamma^*) \to \mathcal{P}(\Sigma^*)$ as
	\[
		f(L)
		=
		\{
		x_{w_1} x_{w_2} \cdots x_{w_k} \in \Sigma^*
		\mid
		w_1 w_2 \cdots w_k \in L
		\text{ where each }
		w_i\in W
		\}.
	\]
	Then, there is a string transducer $M = (\Gamma,\Sigma, Q, A, q_0, \delta)$ such that $f(L) = f(L\cap W^*) = M(L)$.
    Moreover, such a string transducer is computable from the set $W$ and words $x_w\in \Sigma^*$.
\end{lemma}

The proof below is standard. In particular, the existence of such an automaton is known since $W$ is a code (cf.~\cite{berstel1985theory}), see for example \cite[p.~200-2]{berstel1985theory} in which a similar automaton is constructed and called a `decoding automaton'.

\begin{proof}
	Firstly we observe that if $W = \emptyset$, then $f(L) = \emptyset$ for each language $L\subseteq \Gamma^*$.
	In this case, any such string transducer with $A = \emptyset$ satisfies the statement of the lemma.
	Thus, in the remainder of this proof, we assume that $W\neq \emptyset$.

	Let $w\in \Gamma^*$ be a word for which $w\in W^*$.
	Then, since $W$ is a finite antichain with respect to the prefix order, there is a unique factorisation of $w$ as
	$
		w =w_1 w_2 \cdots w_k
	$
	where each $w_i\in W$.

	We construct a string transducer $M = (\Gamma,\Sigma, Q, A, q_0, \delta)$ as follows.
	For each proper prefix $u \in \Gamma^*$ of a word $w\in W$, we introduce a state $q_u \in Q$.
	The initial state is $q_0 =q_\varepsilon$, and the set of accepting states is $A = \{q_\varepsilon\}$.
	Further, our automaton has one additional state $q_\mathrm{fail}$ which is a fail state; that is,
	\[
		\delta(g, q_\mathrm{fail}) = (\varepsilon, q_\mathrm{fail})
	\]
	for each $g\in \Gamma$.
	We then specify the remaining transitions as follows.

	For each state $q_u$ with $u\in \Gamma^*$, and each $g\in \Gamma$, we define the transition
	\[
		\delta(g, q_u)
		=
		\begin{cases}
			(x_w, q_\varepsilon)           & \text{if } w = ug \in W,                                \\
			(\varepsilon, q_{ug})          & \text{if } ug\text{ is a proper prefix of some }w\in W, \\
			(\varepsilon, q_\mathrm{fail}) & \text{otherwise}.
		\end{cases}
	\]
	The string transducer $M$ is now completely specified.
	It is clear from the construction that $f(L)=M(L)$ for each $L\subseteq\Gamma^*$.
    Moreover, one observes that every step of this construction is computable.
\end{proof}

\begin{lemma}\label{lem:example-antichain}
	Suppose that $G$ is an infinite group with a finite monoid generating set $X$.
	Fix a finite number of words $u_1,u_2,\ldots,u_k\in X^*$.
	Then there exists a choice of non-empty words $w_1,w_2,\ldots,w_k\in X^*\setminus \{\varepsilon\}$, such that each $\overline{w_i}=1$ and
	\[
		W = \{
		w_1 u_1, w_2 u_2,\ldots, w_k u_k
		\}
	\]
	is an antichain in prefix order; that is, for each choice of words $x,y\in W$, the word $x$ is not a proper prefix of $y$.
    Moreover, such a choice of set $W$ is computable if the word problem for $G$ is computable.
\end{lemma}

\begin{proof}
	We begin by constructing the words $w_1,w_2,\ldots,w_k$ as follows.

	Let $\alpha_1 \in X$ be a nontrivial generator, that is, $\overline{\alpha_1}\neq 1$; then let $\beta_1 \in X^*$ be a geodesic with $\overline{\alpha_1\beta_1} = 1$ (If $X$ is a symmetric generating set, then we may choose $\beta_1 = \alpha_1^{-1}$.)
	From this selection, we  define $w_1 = \alpha_1\beta_1$.
	We now choose the words $w_2,w_3,\ldots,w_k$ sequentially as follows.

	For each $i\geq 2$, we choose a geodesic $\alpha_i\in X^*$ with length $|\alpha_i|= |\alpha_{i-1}\beta_{i-1}|+1$.
	We then choose a geodesic $\beta_i\in X^*$ such that $\overline{\alpha_i\beta_i}=1$ (If $X$ is a symmetric generating set, then we may choose $\beta_i = \alpha_i^{-1}$.)
	Then, define $w_i = \alpha_i\beta_i$.

	We have now selected a sequence of words $w_1,w_2,\ldots,w_k$.
	For each word $w_i$, let $\gamma_i$ denote the longest prefix which is a geodesic.
	Then
	\[
		|\gamma_{i-1}| < |w_{i-1}| < |\alpha_{i}| \leq |\gamma_{i}|
	\]
	for each $i\in\{2,3,\ldots,k\}$.
	Thus, $|\gamma_1|< |\gamma_2|< \cdots< |\gamma_k|$ and $|\gamma_i| < |w_iu_i|$ for each $i$.

	We now see that, if $w_i u_i$ is a proper prefix of some word $v\in X^*$, then $\gamma_i$ is also the longest prefix of $v$ which is a geodesic.
	Hence, we conclude that the set
	\[
		W = \{
		w_1 u_1, w_2 u_2,\ldots, w_k u_k
		\}
	\]
	is an antichain as required.
    All computations are possible as long as the word problem is computable for $G$.
\end{proof}

\begin{proposition}\label{prop:closure-under-gset}
Suppose that the group $G$ has a finite symmetric generating set $X$, and that $M\subseteq G$ is a subset of the group for which
$
    L = \{
        w\in X\star
    \mid
        \overline{w}\in M
    \}
$
is an unambiguous limiting ET0L language.
Then, for each subgroup $H \leq G$ with finite generating set $Y \subseteq H$, the set
$
    L'=\{
        w \in Y\star
    \mid
        \overline{w} \in M \cap H
    \}
$
forms an unambiguous limiting ET0L language.
Moreover, an unambiguous limiting ET0L grammar for $L'$ is computable from such a grammar for $L$.
\end{proposition}

\begin{proof}
We write the set $Y = \{y_1,y_2,...,y_m\}$. Then, for each $y_i\in Y$, fix a non-empty word $u_i\in X^*$ for which $\overline {y_i} = \overline{u_i}$.
From \cref{lem:example-antichain}, we then see that there is a choice of words $w_1,w_2,...,w_m$, with each $\overline{w_i}=1$, such that the words $w_1 u_1$, $w_2 u_2$, ..., $w_m u_m$ form an antichain with respect to prefix order.
From our choice of words $u_i$ and $w_i$, we then see that
\[
    L'
    =
    \left\{
        y_{i_1}
        y_{i_2}
        \cdots
        y_{i_k} \in Y^*
    \middle|
        (w_{i_1} u_{i_1})
        (w_{i_2} u_{i_2})
        \cdots
        (w_{i_k} u_{i_k})
        \in L
    \right\}.
\]
Thus, from \cref{lem:f_is_detfsm}, we see that there is a string transducer $M$ for which $L' = M(L)$.
Moreover, we see that this string transducer is injective. Our result then follows from \cref{prop:string-transducer}.
Moreover, we see that every step of our construction is computable.
\end{proof}

\subsection{Generating Functions}\label{sec:gfun}

A multivariate generating function is a generalisation of the ordinary generating function to several variables. It is used to study sequences indexed by multiple indices, such as $a_{i_1, i_2, \ldots, i_k}$.
Formally, for a $k$-dimensional sequence $\{a_{i_1, i_2, \ldots, i_k}\}$, the multivariate generating function is defined as
\[
g(x_1, x_2, \ldots, x_k) = 
\sum_{i_1=0}^{\infty} \sum_{i_2=0}^{\infty} \cdots \sum_{i_k=0}^{\infty}
a_{i_1, i_2, \ldots, i_k} \, x_1^{i_1} x_2^{i_2} \cdots x_k^{i_k}.
\]
Multivariate generating functions provide a compact way to encode and manipulate multidimensional sequences. They are especially useful in combinatorics, probability, and the study of systems with several interacting parameters (see, \cite{book}). The set of such functions is equipped with a natural notion of convergence. 
Let $\{g_n(x_1, \ldots, x_k)\}_{n \ge 0}$ be a sequence of multivariate generating functions
\[
g_n(x_1, \ldots, x_k)
= \sum_{i_1, \ldots, i_k \ge 0} a^{(n)}_{i_1, \ldots, i_k}
  x_1^{i_1} \cdots x_k^{i_k}.
\]
Then  $g_n \to g$ if, for each $i_1,i_2,\ldots,i_k$, we have 
  $a^{(n)}_{i_1, \ldots, i_k} \to a_{i_1, \ldots, i_k}. $
  
We are now ready to prove our theorem on the generating functions of ET0L languages as follows.

\TheoremGFun

\begin{remark}
    The coefficients of the generating functions $g_n$ correspond to the number of words in the set $S \cdot \alpha \beta^n \in (\Sigma\cup V)^*$.
    We know that the count of all such words converges from properties 4 and 5 of \cref{def:limiting-et0l}.
    Hence, the limit in the statement of the theorem exists.
\end{remark}

\begin{proof}
Let $E = (\Sigma,V,T,\mathcal{R},S)$ be an ET0L language as in \cref{def:limiting-et0l} where $V = \{X_1=S,X_2,...,X_k\}$, $T = \{\alpha,\beta,\gamma\}$ and $\mathcal{R} = \alpha\beta\star\gamma$.
Suppose that
\[
    \alpha\colon
    \left\{
    \begin{aligned}
        X_1 &\mapsto L_{\alpha,1}\\
        X_2 &\mapsto L_{\alpha,2}\\
        &\rvdots\\
        X_k &\mapsto L_{\alpha,k}
    \end{aligned}
    \right.
    ,\qquad
    \beta\colon
    \left\{
    \begin{aligned}
        X_1 &\mapsto L_{\beta,1}\\
        X_2 &\mapsto L_{\beta,2}\\
        &\rvdots\\
        X_k &\mapsto L_{\beta,k}
    \end{aligned}
    \right.
    \qquad\text{and}\qquad
    \gamma\colon
    \left\{
    \begin{aligned}
        X_1 &\mapsto L_{\gamma,1}\\
        X_2 &\mapsto L_{\gamma,2}\\
        &\rvdots\\
        X_k &\mapsto L_{\gamma,k}
    \end{aligned}
    \right.
\]
where each $L_{\alpha,i}$, $L_{\beta,i}$ and $L_{\gamma,i}$ is a regular language over $V\cup \Sigma$.

In the remainder of this proof, we write $\mathbf{x}$ for the tuple of variables $(x_1,x_2,...,x_k)$ where each variable $x_i$ corresponds to the variable $X_i \in V$.
Suppose that $m=|\Sigma|$, we then write $\mathbf{y}$ for the tuple of variables $(y_1,y_2,...,y_m)$ where each variable $y_i$ corresponds to a letter $\sigma_i$ in $\Sigma = \{\sigma_1, \sigma_2, ...,\sigma_m\}$.

For the regular languages $L_{\alpha,i}$, $L_{\beta,i}$ and $L_{\gamma,i}$ we write $h_{\alpha,i}(\mathbf{x},\mathbf{y})$, $h_{\beta,i}(\mathbf{x},\mathbf{y})$ $h_{\gamma,i}(\mathbf{x},\mathbf{y})$ for their generating functions, respectively.
It is well-known that the multivariate generating function of regular languages are rational (this follow from \cite[p.~125]{salomaa1990formal}).

Then, for each $h_{\gamma,i}$, we define a function $H_{\gamma,i}(z)$ as
\[
    H_{\gamma,i}(z)
    =
    h_{\gamma,i}(
            \underbrace{0,0,...,0}_{k\text{ times}},
            \underbrace{z,z,...,z}_{m\text{ times}}
        ).
\]
We then see that the generating function can be written as
\[
    f(z) = g(
        H_{\gamma,1}(z),
        H_{\gamma,2}(z),...,
        H_{\gamma,k}(z),
        z,z,...,z
    )
\]
with
\[
    g(\mathbf{x},\mathbf{y})
    =
    \lim_{n\to\infty} g_n(\mathbf{x},\mathbf{y})
\]
where
\begin{align*}
    g_0(\mathbf{x},\mathbf{y})&= h_{\alpha,1}(\mathbf{x},\mathbf{y})
    \quad
    \text{ and }
    \\g_{n+1}(\mathbf{x},\mathbf{y}) &= g_{n}(
        h_{\beta,1}(\mathbf{x},\mathbf{y}),
        h_{\beta,2}(\mathbf{x},\mathbf{y}),
        ...,
        h_{\beta,k}(\mathbf{x},\mathbf{y}),
        \mathbf{y}
    )
\end{align*}
for each $n \geq 0$.
\end{proof}

\section{Main theorem}\label{sec:main}

The content in this section is devoted to proving our main theorem, stated as follows.

\TheoremMain

We begin by noting that, from \cref{prop:closure-under-gset,lem:fg.ss.ba_overgroup}, it is sufficient to consider the case of $G$  a bounded automata group with a finite generating set which can be partitioned as $X = F\cup D$ where $F \subset \FinTd$ is a set of finitary automorphisms and $D \subset \DirTd$ is a set of directed automorphisms.
Thus, in the remainder of this section, we will assume so without loss of generality.

In the theorem statement, we are given an eventually periodic ray $\eta = a b^\omega$, specified as two finite-length words $a,b\in C^*$.
The goal is to construct two ET0L grammars, $E = (X,V,T,\mathcal R, S)$ and $E' = (X,V,T,\mathcal R, S')$, for the languages $\WP(G,X,\Stab(\eta))$ and $X^*\setminus \WP(G,X,\Stab(\eta))$, respectively.
Moreover, at the end of our construction, we point out that these grammars are unambiguous limiting.
The only difference between the two grammars is their starting symbol.

In order to describe the nonterminals in our grammar, we need to define a finite set of eventually periodic rays  described by a finite set of pairs $I \subset C^*\times C^*$ as follows.

\begin{lemma}\label{lem:setI}
Let $I\subseteq C^* \times C^*$ be a finite set  such that, for each
\[
    \zeta \in \{\eta\} \cup \mathrm{spine}(D)\cup (\mathrm{spine}(D)\cdot D)
\]
where $\mathrm{spine}(D) = \{ \mathrm{spine}(\delta)\mid \delta\in D\}$,
there exists some $(u,v)\in I$ such that $\zeta = uv^\omega$.
Then it is effectively computable to construct such a set $I$, where additionally
\begin{enumerate}
    \item there exists some $\ell \in \mathbb N$ such that $|u|=|v|=\ell$ for each $(u,v)\in I$;
    \item if $(u,v),(u',v')\in I$ with $(u,v)\neq (u',v')$, then $u\neq u'$;
    \item for each $x\in D$, there is some $(u,v)\in I$ such that $\mathrm{spine}(x) = u v^\omega$ with $x\at u = x\at uv$; and
    \item for each finitary automorphism $f\in F$, we have $\ell \geq \mathrm{depth}(f)$.
\end{enumerate}
\end{lemma}

\begin{proof}
For the moment, if we ignore properties 1--4, it is clear that we may construct such a set $I$ from a description of the generating set of bounded automaton automorphisms $X$.
In particular, given generators $x,y\in X$, it is computable to check if they are not finitary (i.e.~there are no non-trivial cycles among the restrictions in their descriptions), and it is also computable to find the initial and periodic segment of their spines (it follows from the proof of \cref{lemma:spine is eventualy periodic}).
Moreover, the action of $y$ on the spine of $x$ can then be computed using a finite amount of memory.

In the remainder of this proof, we show how to modify such a set so that it satisfies each of the desired properties of the lemma.

\medskip
\noindent
\underline{Property 1}:
Let $m\geq 0$ be the constant defined as
\[
    m
    =
    \max\{
        |u|
    \mid
        (u,v) \in I
    \}.
\]
Now let $\ell\geq 1$ be defined as
\[
    \ell = \mathrm{lcm}(
        \{|v| \mid (u,v)\in I\}
        \cup
        \{m\}
    ).
\]
We now construct a finite set $I'$ as follows.
For each $(u,v)\in I$, we define $u'$ as the length-$\ell$ prefix of $uv^\omega$, and we define $v'$ to be a cyclic permutation of the word $v$ for which $u'(v')^\omega = uv^\omega$.
We add $(u',(v')^{\ell/|v'|})$ to the set $I'$.

We now see that if $(a,b)\in I'$, then $|a|=|b|=\ell$ and that $I'$ represents the same elements as $I$.
We also see that $I'$ is minimal since for each  $(a,b),(a',b')\in I'$, we have $ab^\omega = a'(b')^\omega$ if and only if $(a,b)=(a',b')$.
The above steps are computable.
Thus, given a set $I$, there is an algorithm that can construct the set $I'$.

We replace $I$ with $I'$, then in the remainder of this proof, we assume that there is some $\ell\geq 1$ such that for each $(u,v) \in I$, we have $|u|=|v|=\ell$.

\medskip
\noindent
\underline{Property 2}:
For each $(u,v),(u',v')\in I$ with $(u,v)\neq (u',v')$, we have $uv\neq u'v'$.
We construct a set $I'\subset C^*\times C^*$ as
\[
    I'
    =
    \{
        (uv,vv)
    \mid
        (u,v)\in I
    \}.
\]
We note that the set $I'$ has properties 1 and 2.
Moreover, given a set $I$, there is an algorithm which can construct the set $I'$.
Thus, we replace $I$ with $I'$, and $\ell$ with $2\ell$ we may assume that our set $I$ satisfies properties~1 and~2.

\medskip
\noindent
\underline{Properties 3}:
By \cref{lemma:spine is eventualy periodic}, for each $x\in D$, we can choose words $u_x,v_x \in C^*$ such that $\mathrm{spine}(x)=u_x v_x^\omega$ with $x\at u_x = x\at u_x v_x$.
\smallskip

\noindent
\medskip
\noindent
\underline{Properties 4}:
We define a constant $\ell'\geq 1$ as
\[
    \ell'
    =
    \mathrm{lcm}
    \left(
            \{ \ell \}
            \cup
            \{ |u_x| \mid x\in D \}
            \cup
            \{ |v_x| \mid x\in D \}
            \cup
            \{ \mathrm{depth}(f) \mid f\in F \}
    \right).
\]
We now define a set $I'$ as
\[
    I'
    =
    \{
        (uv^{(\ell'/\ell) - 1}, v^{\ell'/\ell})
    \mid
        (u,v)\in I
    \}.
\]
We then see that this new set $I'$ satisfies the properties~1 and~2, and also satisfies property~3.

Observe that the above steps are computable, that is, given a set $I$ and a description of the generators $D$, there is an algorithm which can construct the set $I'$.
Thus, after replacing $I$ with $I'$ and $\ell$ with $\ell'$, we may assume that $I$ satisfies all of our desired properties.
\end{proof}

Our grammars have three tables $T = \{\tau_\mathrm{init}, \tau_\mathrm{up}, \tau_\mathrm{finish}\}$ and rational control $\mathcal{R} = \tau_\mathrm{init} (\tau_\mathrm{up})^* \tau_\mathrm{finish}$.

The construction of the tables will be based on directional depth  as defined in \cref{def:compdepth}. To proceed with it, we introduce the following tool.

\begin{definition}\label{def:decorated-words}
Let $w = s_1s_2\cdots s_k \in X^*$, where each $s_i \in X$, and let $\zeta\in C^*\cup C^\omega$.

Define the word $w$ \emph{decorated with respect to} $\zeta$ as
\[
    \Decorate(\zeta,w)\coloneqq
    s_1^{(a_1)}
    s_2^{(a_2)}
    \cdots
    s_k^{(a_k)}
\quad\text{ where }\quad
    a_i = \CompDepth(\zeta_i, s_i) 
\]
with $\zeta_1 = \zeta$ and $\zeta_{i+1} = \zeta_{i}\acts s_{i}$ for each $i\in \{1,2,...,k\}$.
\end{definition}

Before proceeding with the construction, let us give an overview of how it is intended to work.
Suppose $w = s_1 s_2\cdots s_k$ is a word generated by the grammar $E$ with
\begin{equation}\label{eq:generating-word}
    \Decorate(\zeta,w)\coloneqq
    s_1^{(a_1)}
    s_2^{(a_2)}
    \cdots
    s_k^{(a_k)}.
\end{equation}
Then, when producing  $w$, our grammar fills in the letters $s_i$ in decreasing order of $a_i$.
In particular, let $\ell\in \N$ be the constant derived in \cref{lem:setI}, and suppose that $A\in \N$ is chosen such that $\lceil a_i/\ell\rceil \leq A+3$ for each finite $a_i < \infty$.
Then, our grammar will produce the word $w$ as
\[
    w \in S \cdot
        \tau_\mathrm{init} \,
        \left(\tau_\mathrm{up}\right)^{A}
        \tau_\mathrm{final}.
\]
In particular: \textit{(note that in the following $a_i$ are as in (\ref{eq:generating-word}))}
\begin{itemize}
    \item The letters $s_i$ with $a_i = \infty$ are generated when the table $\tau_\mathrm{init}$ is applied.
    \item The letters $s_i$ with $3\ell < a_i < \infty$ first come into a sentential form after applying the last $\tau_\mathrm{up}$ of
    \begin{equation}\label{eq:table1}
        \tau_\mathrm{init}
        \left(\tau_\mathrm{up}\right)^{A+4 - \lceil a_i / \ell \rceil},
    \end{equation}
    that is, for each $m$, letters $s_i$ with $(3+m)\ell < a_i \leq (4+m)\ell$ enter a sentential form after applying
    \[
        \tau_\mathrm{init}
        \left(\tau_\mathrm{up}\right)^{A - m}.
    \]
    In (\ref{eq:table1}) we have $A+4$ in the exponent as in this case we have $4\leq \lceil a_i /\ell\rceil\leq A+3$, and thus it would follow that $1\leq A+4-\lceil a_i/\ell\rceil \leq A$, i.e., the sequence of tables contains at least one $\tau_\mathrm{up}$.
    \item The letters $s_i$ with $a_i \leq 3\ell$ are generated at the end of the production by the table $\tau_\mathrm{final}$.
\end{itemize}
Our construction of the grammar $E'$ satisfies analogous properties.
During each part of this production, we ensure that words are produced unambiguously with respect to their rational control.

\subsection {Nonterminals and starting symbols}

We begin by introducing the starting symbols $S = \llbracket \eta ; \eta \rrbracket$ and $S' = \llbracket \eta ; \neg\eta \rrbracket$ of the grammars $E$ and $E'$, respectively.
The nonterminal $S$ is a placeholder for a word whose action stabilises the ray $\eta$.
Similarly, the nonterminal $S'$ is a placeholder for a word whose action does not stabilise $\eta$, that is, whose action takes $\eta$ to a different ray.
We now define the remaining nonterminals as follows.

In the following, it should be understood that $I$ refers to a set of pairs as constructed in \cref{lem:setI}.

\subsubsection{Nonterminals with restrictions on result of a word action}
\label{sec:nonterminals}

Each of the nonterminals which we introduce below can be a placeholder for words which have very particular actions on a particular vertex of the tree $\Td$.
The main idea of our construction of the grammar $E$ is to decompose the action of a word on $\eta$ into a sequence of actions of this form.

For each pair of rays $\zeta=uv^\omega,\zeta'=u'(v')^\omega$ where $(u,v),(u',v')\in I$ and each pair of paths $p,p'\in C^{2\ell}$, we have nonterminals $\llbracket \zeta, p ; \zeta',p'\rrbracket$ and $\llbracket \zeta,p; \neg \eta \rrbracket$.

In our proof, we construct our grammars so that for each $m \geq 0$, we have
\[
    w\in X^*\cap
    (
    \llbracket \zeta, p ; \zeta',p'\rrbracket
    \cdot
    (\tau_\mathrm{up})^m \tau_\mathrm{finish}
    )
\]
if and only if $u v^m p\acts \overline{w} = u'(v')^m p'$ and
\[
    \Decorate(uv^mp, w)
    =
    s_1^{(a_1)} s_2^{(a_2)}\cdots s_k^{(a_k)}
\]
where each $a_i \leq |uv^mp| = (m+3)\ell$.
Moreover, our grammar ensures that, for each $m \geq 0$, we have
\[
    w\in X^*\cap (
    \llbracket \zeta, p ; \neg \eta\rrbracket
    \cdot
    (\tau_\mathrm{up})^m \tau_\mathrm{finish})
\]
if and only if $u v^m p\acts \overline{w} \neq ab^{m+2}$ with $(a,b)\in I$ and $ab^\omega = \eta$, and that
\[
    \Decorate(uv^mp, w)
    =
    s_1^{(a_1)} s_2^{(a_2)}\cdots s_k^{(a_k)}
\]
where each $a_i \leq |uv^mp| = (m+3)\ell$.
That is, $\llbracket \zeta, p ; \neg \zeta\rrbracket$ corresponds to words whose action does not take words of the form $uv^mp$ to a prefix of $\eta$.

Let $ab^\omega=\eta$ where $(a,b)\in I$, then we see that the nonterminals $\llbracket\eta,b^2;\eta,b^2\rrbracket$ and $\llbracket \eta,b^2; \neg\eta\rrbracket$ can be used to generate subwords of words in $\WP(G,X,\Stab(\eta))$ and in $X^*\setminus \WP(G,X,\Stab(\eta))$ respectively, which are composed of letters $s_j$ with $a_j< \infty$ as in (\ref{def:decorated-words})

\subsubsection
[Additional nonterminals for E']
{Additional nonterminals for $E'$}
\label{sec:nonterminals-additional}

For technical reasons, for the grammar $E'$, we require some additional nonterminals which are placeholders for 
factors of output words containing letters $s_j$ with each $a_j< \infty$ as in (\ref{def:decorated-words}).

For each ray $\zeta=uv^\omega$ with $(u,v)\in I$ and each path $p\in C^{2\ell}$, we introduce a nonterminal $\llbracket \zeta,p; \mathrm{any}\rrbracket$.
We construct our grammar such that for each $m\geq 0$, we have
\[
    w \in X^* \cap ( \llbracket \zeta,p;\mathrm{any}\rrbracket\cdot (\tau_\mathrm{up})^m \tau_\mathrm{finish})
\]
if and only if
\[
    \Decorate(uv^mp, w)
    =
    s_1^{(a_1)} s_2^{(a_2)}\cdots s_k^{(a_k)}
\]
where each $a_i \leq |uv^mp| = (m+3)\ell$.

\subsubsection{Intuition of the starting symbols}\label{sec:intuition-of-starting}

In our proof, we construct the tables of our grammar in such a way that, for each $m\geq 0$, we have
\[
    w\in X^*\cap ( \llbracket \eta; \eta\rrbracket\cdot\tau_\mathrm{init}(\tau_\mathrm{up})^m \tau_\mathrm{finish})
\]
if and only if $w\in \mathcal \WP(G,X,\Stab(\eta))$ and
\[
    \Decorate(\eta, w)
    =
    s_1^{(a_1)} s_2^{(a_2)}\cdots s_k^{(a_k)}
\]
where each $a_i \in \{0,1,2,3,...,(m+3)\ell\}\cup \{\infty\}$.
For each $m\geq 0$,
\[
    w\in X^*\cap ( \llbracket \eta; \neg\eta\rrbracket\cdot\tau_\mathrm{init}(\tau_\mathrm{up})^m \tau_\mathrm{finish})
\]
if and only if $w\in X^*\setminus \WP(G,X,\Stab(\eta))$ and
\[
    \Decorate(\eta, w)
    =
    s_1^{(a_1)} s_2^{(a_2)}\cdots s_k^{(a_k)}
\]
where each $a_i \in \{0,1,2,3,...,(m+3)\ell\}\cup \{\infty\}$.
From these properties, it is clear that the grammars $E$ and $E'$ satisfy the limiting property (\ref{def:limiting-et0l/limit}) as in \cref{def:limiting-et0l}.
Moreover, from our choice of tables and rational control, we see that these grammars satisfy properties (\ref{def:limiting-et0l/1}) and (\ref{def:limiting-et0l/2}) in \cref{def:limiting-et0l}.
In the remainder of our construction, we verify that it indeed satisfies the remaining properties of \cref{def:limiting-et0l}.

\subsubsection{Computability of nonterminals}

Given the set $I\subset C^*\times C^*$ and a description of the ray $\eta = ab^\omega$, we can list all of the finite nonterminals.

\subsection
[Initialisation map]
{Initialisation map: $\tau_\mathrm{init}$}
\label{sec:map-init}

In order to explain the constructions of the initialisation tables of $E$ and $E'$, we begin with the following observation.
Let $w = s_1 s_2 \cdots s_k \in X^*$ and consider its decorated version with respect to $\eta$
\begin{equation}\label{eq:decorated-accepted-word}
    \Decorate(\eta,w)
    =
    s_1^{(a_1)} s_2^{(a_2)} \cdots s_k^{(a_k)}
\end{equation}
with $a_i \in \mathbb{N}\cup \{\infty\}$.
Then factor the word $w$ uniquely as
\begin{equation}\label{eq:initial-table}
    w = z_0 x_1 z_1 x_2 z_2 \cdots x_q z_q
\end{equation}
where each $z_j \in X^*$, each $x_j\in D$ where the words $z_j$ correspond to (potentially empty) sequences of letters $s_i$ for which $a_i< \infty$, and each $x_j$ correspond to letters $s_i$ for which $a_i = \infty$.
We define our initialisation table such that it fills in each $x_j$ with a directed automorphism, and puts an appropriate placeholder nonterminal in the spot of each word $z_j$.

\subsubsection
[Initialising the grammar E]
{Initialising the grammar $E$}\label{sec:initialise-E}

Let $\eta = ab^\omega$ where $(a,b)\in I$, then from the factorisation in (\ref{eq:initial-table}) we define the language $\llbracket \eta, \eta\rrbracket\cdot\tau_\mathrm{init}$ to contain all words of the form
\begin{multline*}
        \llbracket \eta, b^2 ; \alpha_1, (v_1)^2 \rrbracket 
        x_1
        \llbracket \alpha_1', (v_1')^2 ; \alpha_2, (v_2)^2 \rrbracket
        x_2
        \llbracket \alpha_2', (v_2')^2 ; \alpha_3, (v_3)^2 \rrbracket
        \\
        x_3
        \llbracket \alpha_3', (v_3')^2 ; \alpha_4, (v_4)^2 \rrbracket
        \cdots
        x_k
        \llbracket \alpha_k', (v_k')^2 ; \eta, b^2 \rrbracket
    \in
    (\llbracket \eta, \eta\rrbracket\cdot \tau_\mathrm{init})
\end{multline*}
where
\begin{enumerate}
\item\label{req:word-factor/item1}
$\mathrm{spine}(x_i) = \alpha_i = u_i (v_i)^\omega$ where $(u_i,v_i)\in I$
for each $i \in \{ 1,2,...,k \}$; and
\item\label{req:word-factor/item2}
$\mathrm{spine}(x_i)\cdot x_i = \alpha_i' = u_i' (v_i')^\omega$ 
where $(u_i',v_i')\in I$
for each $i \in \{ 1,2,...,k \}$.
\item\label{req:word-factor/item3}
$b = v_1 = v_k'$ and $v'_i = v_{i+1}$ for each $i\in \{1,2,...,k-1\}$.
\end{enumerate}
Each nonterminal of the form $\llbracket -,-;-,-\rrbracket$ as above corresponds to some word $z_i$ as in (\ref{eq:initial-table}).

Item~\ref{req:word-factor/item3} ensures that the action of the letters $x_i$ are being tracked correctly, in particular, the words corresponding to the placeholders $\llbracket -,-;-,-\rrbracket$ cannot modify the ray beyond a particular finite depth: Item~\ref{req:word-factor/item3} ensures that the tail of these rays match.

The set of words $\llbracket \eta,\eta \rrbracket \cdot \tau_\mathrm{init} \subseteq (X\cup V)^*$, as defined above, is a regular language since, for each word in the language, the possible values of each letter depends, at most, on the previous two nonterminals.
Thus, one could construct a finite-state automaton to recognise all such words where the states of the automaton correspond to the possible values of the two previous nonterminals of the form $\llbracket -,-;-,-\rrbracket$.

\subsubsection
[Initialising the grammar E']
{Initialising the grammar $E'$}\label{sec:initialise-E'}

Let $\eta = ab^\omega$ where $(a,b)\in I$, then from the factorisation in (\ref{eq:initial-table}) we define the language $\llbracket \eta,\neg\eta\rrbracket\cdot\tau_\mathrm{init}$ to contain all words of the form
\begin{multline*}
        \llbracket \eta, b^2 ; \alpha_1, (v_1)^2 \rrbracket 
        x_1
        \llbracket \alpha_1', (v_1')^2 ; \alpha_2, (v_2)^2 \rrbracket
        x_2
        \llbracket \alpha_2', (v_2')^2 ; \alpha_3, (v_3)^2 \rrbracket
        \\
        x_3
        \llbracket \alpha_3', (v_3')^2 ; \alpha_4, (v_4)^2 \rrbracket
        \cdots
        x_k
        \llbracket \alpha_k', (v_k')^2 ; \underline{\varphi} \rrbracket
    \in
    (\llbracket \eta, \neg\eta\rrbracket\cdot\tau_\mathrm{init})
\end{multline*}
where
\begin{enumerate}
\item\label{req:word-factor2/item1}
$\mathrm{spine}(x_i) = \alpha_i = u_i(v_i)^\omega$ where $(u_i,v_i)\in I$
for each $i \in \{ 1,2,...,k \}$;
\item\label{req:word-factor2/item2}
$\mathrm{spine}(x_i)\cdot x_i = \alpha_i' = u_i' (v_i')^\omega$ where $(u_i',v_i')\in I$
for each $i \in \{ 1,2,...,k \}$;

\item\label{req:word-factor2/item3}
$b = v_2$ and $v'_i = v_{i+1}$ for each $i\in \{1,2,...,k-1\}$; and

\item\label{req:word-factor2/item4}
if $v'_k = b$, then $\underline{\varphi} = \neg \eta$, otherwise, $v'_k\neq b$ and $\underline{\varphi} = \mathrm{any}$.
\end{enumerate}
Each nonterminal in the above corresponds to some factor $z_i$ as in (\ref{eq:initial-table}).
Items~\ref{req:word-factor2/item1} and~\ref{req:word-factor2/item2} ensure that it is possible for each $x_i$ to have infinite directional depth; item~\ref{req:word-factor2/item3} and~\ref{req:word-factor2/item4} ensure that the word has an action which does not stabilise the ray $\eta$, and that this action is being correctly tracked.

The set $\llbracket \eta, \neg\eta \rrbracket\cdot \tau_\mathrm{init} \subseteq (X\cup V)^*$ is regular for precisely the same reasons as $\llbracket \eta,\eta\rrbracket\cdot \tau_\mathrm{init}$ is regular in \cref{sec:initialise-E}.
That is, each possible choice of nonterminal depends, at most, on the previous two nonterminals of the form $\llbracket -,-;-,-\rrbracket$.

\subsubsection{Validity and computability}

From our definition of $\tau_\mathrm{init}$, we see that both $\llbracket \eta,\eta\rrbracket\cdot\tau_\mathrm{init}$ and $\llbracket \eta,\neg\eta\rrbracket\cdot\tau_\mathrm{init}$ are regular languages (cf.~\cref{sec:initialise-E,sec:initialise-E'}), and thus $\tau_\mathrm{init}$ is a table as in \cref{def:tables}.
For all other nonterminals, the table $\tau_\mathrm{init}$ does not need to be specified, as this table will only be applied to the starting symbols.

Given a set $I\subset C^*\times C^*$ and a set of nonterminals, there is an algorithm which can generate the finite state automata for the table $\tau_\mathrm{init}$.

\subsection
[Processing map]
{Processing map: $\tau_\mathrm{up}$}

We now describe the map $\tau_\mathrm{up}$ that performs replacements on the nonterminals of the form $\llbracket \zeta,p; \zeta',p'\rrbracket$, $\llbracket \zeta,p ; \neg\eta\rrbracket$ and $\llbracket \zeta,p ; \mathrm{any} \rrbracket$.
Let $\zeta = uv^\omega$ where $(u,v)\in I$.
To simplify the explanation of this map, we begin by giving a sketch of the intended meaning of each such nonterminal.

Recall from \cref{sec:nonterminals,sec:nonterminals-additional} that if the word $w\in X^*$ belongs to the set
\[
    X^*\cap (\llbracket \zeta,p; \zeta',p'\rrbracket\cdot(\tau_\mathrm{up})^m \tau_\mathrm{finish}),\ 
    X^*\cap (\llbracket \zeta,p ; \neg\eta\rrbracket\cdot(\tau_\mathrm{up})^m \tau_\mathrm{finish})
    \ \ \text{or}\ \ 
    X^*\cap (\llbracket \zeta,p ; \mathrm{any} \rrbracket\cdot(\tau_\mathrm{up})^m \tau_\mathrm{finish}),
\]
for some $m\geq 0$, then we can decorate the word $w$ as
\[
    \Decorate(uv^mp,w) = s_1^{(a_1)} s_2^{(a_2)} \cdots s_k^{(a_k)}
\]
where each $a_i \leq |uv^mp| = (m+3)\ell$.
Analogously to (\ref{eq:initial-table}) in \cref{sec:map-init}, $w$ can be uniquely factored as
\begin{equation}\label{eq:processing-factorisation}
    w = z_0 x_1 z_1 x_2 z_2 \cdots x_s z_s
\end{equation}
where each $z_j\in X^*$ contains the letters $s_i$ of the word $w$ for which $a_i \leq (m+2)\ell$, and each $x_j$ is a letter $s_i$ of the word $w$ for which $a_i$ is bounded as $(m+2)\ell < a_i \leq (m+3)\ell$.
From \cref{lem:setI}, we see that each $x_j$ must be a directed automorphism in $D = X\cap \DirTd$ as the corresponding letter $s_i$ has $a_i > \ell$.

We define the map $\tau_\mathrm{up}$ so that it interprets the nonterminal $\llbracket \zeta,p; \zeta',p'\rrbracket$, $\llbracket \zeta,p ; \neg\eta\rrbracket$ and $\llbracket \zeta,p ; \mathrm{any} \rrbracket$ as words of the form (\ref{eq:processing-factorisation}), by producing words where each $z_i$ is represented by some placeholder of the form $\llbracket \zeta,p; \zeta',p'\rrbracket$, $\llbracket \zeta,p ; \neg\eta\rrbracket$ or $\llbracket \zeta,p ; \mathrm{any} \rrbracket$, and each $x_i$ is replaced by an appropriate member of $D = X\cap \DirTd$.
We now construct the table $\tau_\mathrm{up}$ as follows.

\subsubsection
[Case 1]
{Case 1: $\llbracket \zeta,p;\zeta',p'\rrbracket$}
\label{sec:t_up-case1}

Let $\zeta = uv^\omega,\zeta' = u'(v')^\omega$ where $(u,v),(u',v') \in I$, and let $p,p'\in C^{2\ell}$ be paths in the tree.
We then define  $\llbracket \zeta,p;\zeta',p'\rrbracket\cdot\tau_\mathrm{up}$ such that it contains all nonempty words of the form
\[
        \llbracket \zeta, v p_1 ; \alpha_1, q_1 \rrbracket 
        x_1
        \llbracket \alpha_1', q_1' ; \alpha_2, q_2 \rrbracket
        x_2
        \llbracket \alpha_2', q_2' ; \alpha_3, q_3 \rrbracket
        \cdots
        x_k
        \llbracket \alpha_k', q_k' ; \zeta', v' p_1' \rrbracket
    \in
    (\llbracket \zeta,p;\zeta',p'\rrbracket\cdot\tau_\mathrm{up})
\]
where 
\begin{enumerate}
\item\label{item:main-therem/3.1/1}
$v p_1\in C^{2\ell}$ is the length-$2\ell$ prefix of $vp\in C^{3\ell}$;
\item\label{item:main-therem/3.1/2}
$v' p_1' \in C^{2\ell}$ is the length-$2\ell$ prefix of $v' p' \in C^{3\ell}$;
\item\label{item:main-therem/3.1/3}
each $\alpha_i = \mathrm{spine}(x_i) = u_i(v_i)^\omega$ where $(u_i, v_i) \in I$ and $q_i \in C^{2\ell}$;
\item \label{item:main-therem/3.1/4}
each $\alpha_i' = \mathrm{spine}(x_i)\acts x_i = u_i'(v_i')^\omega$ where $(u_i', v_i') \in I$ and $q_i' \in C^{2\ell}$; and
\item\label{item:main-therem/3.1/5}
there is a sequence of words $y_0,y_1,...,y_k \in C^\ell$ such that
\begin{itemize}
    \item $y_0$ is the length-$\ell$ suffix of $p\in C^{2\ell}$,
    \item $y_k$ is the length-$\ell$ suffix of $p'\in C^{2\ell}$,
\end{itemize}
and
\[
    (u_i v_i q_i y_{i-1})\acts x_i = u_i' v_i' q_i' y_i
\]
 for each $i \in \{1,2,...,k\}$ such that
\begin{equation}\label{eq:directional-depth-req}
    4\ell=|u_i v_i q_i|
    <
    \CompDepth (u_i v_i q_iy_{i-1} , x_i) \leq |u_i v_i q_iy_{i-1}|=5\ell
\end{equation}
for each $i\in \{1,2,...,k\}$.
\end{enumerate}

Items (\ref{item:main-therem/3.1/3},\ref{item:main-therem/3.1/4},\ref{item:main-therem/3.1/5}) above imply that for each $m\geq 0$, we have
\[
    (u_i (v_i)^m q_i y_{i-1})\acts x_i
    =
    u_i' (v_i')^m q_i' y_i
\]
 for each $i\in \{1,2,...,k\}$ such that
\[
    (m+3)\ell=|u_i (v_i)^m q_i|
    <
    \CompDepth (u_i (v_i)^m q_i y_{i-1} , x_i) \leq (m+4)\ell
\]
for each $i\in \{1,2,...,k\}$.

The set of words $\llbracket \zeta,p;\zeta',p'\rrbracket\cdot \tau_\mathrm{up}$ forms a regular language as the possible values of each letter depends, at most, on the previous two letters and the previous word of the form $y_i$ as in Item~\ref{item:main-therem/3.1/5} as above.
Thus, we may construct a finite-state automaton to recognise all such words.

The words of this regular language, as described above, exactly correspond to words of the form (\ref{eq:processing-factorisation}).
In particular, the nonterminals correspond to the words $z_j$ in and the letters $x_j$ correspond to the letters $x_j$ in (\ref{eq:processing-factorisation}).
Our restrictions ensure that the letters $x_j$ have the appropriate directional depth in (\ref{eq:directional-depth-req}), and that the action of the associated words matches the action intended by the placeholder $\llbracket \zeta,p;\zeta',p'\rrbracket$.

\subsubsection
[Case 2]
{Case 2: $\llbracket \zeta,p;\neg\eta\rrbracket$}%
\label{sec:t_up-case2}

Let $\zeta = u v^\omega$ where $(u,v) \in I$, and let $p\in C^{2\ell}$ be a path in the tree.
We then define $\llbracket \zeta,p;\neg\eta\rrbracket\cdot\tau_\mathrm{up}$ such that it contains all words of the form
\[
        \llbracket \zeta, v p_1 ; \alpha_1, q_1 \rrbracket 
        x_1
        \llbracket \alpha_1', q_1' ; \alpha_2, q_2 \rrbracket
        x_2
        \llbracket \alpha_2', q_2' ; \alpha_3, q_3 \rrbracket
        x_3
        \llbracket \alpha_3', q_3' ; \alpha_4, q_4 \rrbracket
        \cdots
        x_k
        \llbracket \alpha_k', q_k' ; \underline{\varphi} \rrbracket
    \in
    (\llbracket \zeta,p;\neg \eta\rrbracket\cdot\tau_\mathrm{up})
\]
where
\begin{enumerate}
\item \label{item:main-therem/4.2/1}
$v p_1\in C^{2\ell}$ is the length-$2\ell$ prefix of $vp\in C^{3\ell}$;
\item\label{item:main-therem/4.2/2}
each $\alpha_i = \mathrm{spine}(x_i) = u_i (v_i)^\omega$ where $(u_i, v_i) \in I$, and $q_i \in C^{2\ell}$;
\item \label{item:main-therem/4.2/3}
each $\alpha_i' = \mathrm{spine}(x_i) \acts x_i = u_i' (v_i')^\omega$ where $(u_i', v_i') \in I$, and $q_i' \in C^{2\ell}$;
\item\label{item:main-therem/4.2/4}
there is a sequence of words $y_0,y_1,...,y_k \in C^\ell$ defined such that
\begin{itemize}
    \item $y_0$ is the length-$\ell$ suffix of $p\in C^{2\ell}$,
\end{itemize}
and
\[
    (u_i v_i q_i y_{i-1})\acts x_i
    =
    u_i' v_i' q_i' y_i
\]
for each $i\in \{1,2,...,k\}$ such that
\[
    4\ell=|u_i v_i q_i|
    <
    \CompDepth (u_i v_i q_i y_{i-1} , x_i) \leq |u_i v_i q_i y_{i-1}|=5\ell
\]
for each $i\in \{1,2,...,k\}$; and
\item\label{item:main-therem/4.2/5}
the value of $\underline{\varphi}$ depends on the value of $y_k$, as in item~\ref{item:main-therem/4.2/4}, in particular,
\[
    \underline{\varphi}
    =
    \begin{cases}
        \neg\eta &\text{if }y_k = b\text{ where }\eta=ab^\omega \text{ with } (a,b)\in I\\
        \mathrm{any}&\text{otherwise}.
    \end{cases}
\]
\end{enumerate}
Items~\ref{item:main-therem/4.2/4} and~\ref{item:main-therem/4.2/5} above ensure that the action of the word does not stabilise the ray $\eta$, and that each letter $x_i$ has a directional depth within $\ell$ of the maximum.
In particular, items~(\ref{item:main-therem/4.2/2},\ref{item:main-therem/4.2/3},\ref{item:main-therem/4.2/4}) above imply that, for each $m\geq 0$,
\[
    (u_i (v_i)^m q_i y_{i-1})\acts x_i
    =
    u_i' (v_i')^m q_i' y_i
\]
for each $i\in \{1,2,...,k\}$ such that
\[
    (m+3)\ell=|u_i (v_i)^m q_i|
    <
    \CompDepth (u_i (v_i)^m q_i y_{i-1} , x_i) \leq (m+4)\ell
\]
for each $i\in \{1,2,...,k\}$.

Using the same argument as in \cref{sec:t_up-case1}, we see that the set of words, described above, is a regular language.
In particular, the possible values of each letter can depend on, at most, the previous letters, and on the previous choice of word $y_i$ as in Item~\ref{item:main-therem/4.2/4} as above.
Thus, we can construct a finite-state automaton to recognise all such words.

\subsubsection
[Case 3]
{Case 3: $\llbracket \zeta,p;\mathrm{any}\rrbracket$}%
\label{sec:t_up-case3}

Let $\zeta = uv^\omega$ where $(u,v)\in I$, and let $p\in C^{2\ell}$ be a path in the tree.
We then define $\llbracket \zeta,p;\mathrm{any}\rrbracket\cdot\tau_\mathrm{up}$ such that it contains all words of the form
\[
        \llbracket \zeta, v p_1 ; \alpha_1, q_1 \rrbracket 
        x_1
        \llbracket \alpha_1', q_1' ; \alpha_2, q_2 \rrbracket
        x_2
        \llbracket \alpha_2', q_2' ; \alpha_3, q_3 \rrbracket
        \cdots
        x_k
        \llbracket \alpha_k', q_k' ; \mathrm{any} \rrbracket
    \in
    (\llbracket \zeta,p;\mathrm{any}\rrbracket\cdot\tau_\mathrm{up})
\]
where
\begin{enumerate}
\item \label{item:main-therem/4.3/1}
$v p_1\in C^{2\ell}$ is the length-$2\ell$ prefix of $vp\in C^{3\ell}$;
\item\label{item:main-therem/4.3/2}
each with $\alpha_i = \mathrm{spine}(x_i) = u_i (v_i)^\omega$ where $(u_i, v_i) \in I$, and $q_i \in C^{2\ell}$; 
\item\label{item:main-therem/4.3/3} 
each with $\alpha_i' = \mathrm{spine}(x_i)\acts x_i = u_i' (v_i')^\omega$ where $(u_i', v_i') \in I$, and $q_i' \in C^{2\ell}$; and
\item\label{item:main-therem/4.3/4}
there is a sequence of paths $y_0,y_1,...,y_k \in C^\ell$ defined such that
\begin{itemize}
    \item $y_0$ is the length-$\ell$ suffix of $p\in C^{2\ell}$,
\end{itemize}
and
\[
    (u_i v_i q_i y_{i-1})\acts x_i
    =
    u_i' v_i' q_i' y_i
\]
for each $i\in \{1,2,...,k\}$ such that
\[
    4\ell=|u_i v_i q_i|
    <
    \CompDepth (u_i v_i q_i y_{i-1} , x_i) \leq |u_i v_i q_iy_{i-1}|=5\ell
\]
for each $i\in \{1,2,...,k\}$.
\end{enumerate}
The items~(\ref{item:main-therem/4.3/2},\ref{item:main-therem/4.3/3},\ref{item:main-therem/4.3/4}) imply that for each $m\geq 0$, we have
\[
    (u_i (v_i)^m q_i y_{i-1})\acts x_i
    =
    u_i' (v_i')^m q_i' y_i
\]
for each $i\in \{1,2,...,k\}$ such that
\[
    (m+3)\ell=|u_i (v_i)^m q_i|
    <
    \CompDepth (u_i (v_i)^m q_i y_{i-1} , x_i) \leq (m+4)\ell
\]
for each $i\in \{1,2,...,k\}$.

Using the same argument as in \cref{sec:t_up-case1,sec:t_up-case2}, we see that the set of words, described above, is a regular language.
In particular, the possible values of each letter can depend on, at most, the previous letters, and on the previous choice of word $y_i$ as in Item~\ref{item:main-therem/4.2/4} as above.
Thus, we can construct a finite-state automaton to recognise all such words.

\subsubsection{Computability.}

\smallskip\noindent
Given the set $I\subset C^*\times C^*$ and the set of nonterminals, there is an algorithm that generates the finite state automata described in sections~3.1-3 of this proof.

\subsection
[Final map]
{Final map: $\tau_\mathrm{finish}$}\label{sec:finish-table}

We now complete the description of our tables by constructing the table $\tau_\mathrm{finish}$.
After applying a sequence of tables of the form $\tau_\mathrm{init}(\tau_\mathrm{up})^*$, any word will contain at least one nonterminal which is a placeholder for words with a particular action.
This table attempts to finish the production of words which correspond to factors of $w$ containing letters $s_j$ for which $a_j \leq 3\ell$ as in (\ref{eq:decorated-accepted-word}). 
We note then that if it is not possible to fill in a nonterminal, then the table leaves it unchanged, and thus does not produce a word as output.

Let $(a,b)\in I$ be such that $\eta = a b^\omega$.
For each $\zeta = u v^\omega, \zeta' = u' (v')^\omega$ where $(u,v),(u',v')\in I$ and all paths $p,p'\in C^{2\ell}$, we define the table $\tau_\mathrm{finish}$ as
\begin{align*}
    (\llbracket \zeta,p ; \zeta',p'\rrbracket\cdot\tau_\mathrm{finish} )
    &=
    \left\{
        w \in X^*
    \ \middle|\
    \begin{aligned}
        (up)\acts w=u'p'
        \text{ and }\\
        \Decorate(up, w) = s_1^{(a_1)} s_2^{(a_2)}\cdots s_k^{(a_k)}\\
        \text{where each }a_i\leq |up|
    \end{aligned}
    \right\}
    \cup \{\llbracket \zeta,p ; \zeta',p'\rrbracket\},
    \\
    (\llbracket \zeta,p ; \neg\eta\rrbracket \cdot \tau_\mathrm{finish})
    &=
    \left\{
        w \in X^*
    \ \middle|\
    \begin{aligned}
        (up)\acts w\neq a (b)^2
        \text{ and }\\
        \Decorate(up, w) = s_1^{(a_1)} s_2^{(a_2)}\cdots s_k^{(a_k)}\\
        \text{where each }a_i\leq |up|
    \end{aligned}
    \right\}
    \cup \{\llbracket \zeta,p ; \neg\eta\rrbracket\},
    \\
    (\llbracket \zeta,p ; \mathrm{any}\rrbracket \cdot \tau_\mathrm{finish})
    &=
    \left\{
        w \in X^*
    \ \middle|\
    \begin{aligned}
        \Decorate(up, w) = s_1^{(a_1)} s_2^{(a_2)}\cdots s_k^{(a_k)}\\
        \text{where each }a_i\leq |up|
    \end{aligned}
    \right\}
    \cup \{\llbracket \zeta,p ; \mathrm{any}\rrbracket\}.
\end{align*}
Each of the above is a regular language, in particular, can be recognised by a finite-state automaton with states of the form $q\in C^{3\ell}$, and thus, these automata have at most $d^{3\ell}$ states where $3\ell = |up|$.

In the above definitions, we allow each map to potentially leave a nonterminal unchanged.
We add this possibility so that we satisfy property (\ref{def:limiting-et0l/gamma}) of \cref{def:limiting-et0l}.

\subsection{Proof of main theorem}

We see from our construction that 
\begin{equation*}
    w \in X^* \cap (S\cdot\tau_\mathrm{init} (\tau_\mathrm{up})^m \tau_\mathrm{finish})
\end{equation*}
for some $m \in \N$ if and only if both $w \in \WP(G,X,\Stab(\eta))$ and
\begin{equation}\label{eq:final-decomp}
    \Decorate(\eta,w)
    =
    s_1^{(a_1)} s_2^{(a_2)} \cdots s_k^{(a_k)}
\end{equation}
where each $a_i \in \{0,1,2,...,(m+3)\ell\}\cup \{\infty\}$.
For each $w\in \WP(G,X,\Stab(\eta))$, there exists some $m$, as in (\ref{eq:final-decomp}).
With each application of the tables $\tau_\mathrm{init}$, $\tau_\mathrm{up}$, $\tau_\mathrm{finish}$, we uniquely factor the word $w$ into finitely many subwords.
That is, each word generated by our grammar is generated unambiguously with respect to the rational control.

From \cref{sec:intuition-of-starting} above we see that our grammars satisfy properties (\ref{def:limiting-et0l/1}), (\ref{def:limiting-et0l/2}) and (\ref{def:limiting-et0l/limit}) from \cref{def:limiting-et0l}.
Moreover, from \cref{sec:finish-table} we see that the grammars also satisfy property (\ref{def:limiting-et0l/gamma}) from \cref{def:limiting-et0l}.
Further, from the description of the table $\tau_\mathrm{up}$, we see that they also satisfy property (\ref{def:limiting-et0l/beta}) of \cref{def:limiting-et0l}.

The nonterminals and tables of the grammars are computable.
Thus, we conclude that the grammars $E$ and $E'$ for the languages $\WP(G,X,\Stab(\eta))$ and $X^*\setminus \WP(G,X,\Stab(\eta))$, respectively, are unambiguous limiting ET0L, and are effectively computable.

\section{Are stabilisers of infinite rays context-free?}\label{sec:not-CF} 

Since context-free languages are ET0L, it is natural to ask if \cref{thm:main} can be sharpened to context-free rather than ET0L. Given a bounded automata group $G\leq \AutTd$ with a finite generating set $X$ and $\eta$ any infinite ray in $\Td$, we provide two obstructions to the language $\WP(G,X,\Stab(\eta))$ being context-free. Exploiting results and techniques from \cite{BDN2017}, for a large class of groups, we show that almost all of these languages are not context-free. For a few key examples, we show that all such languages are not context-free. Hence, \cref{thm:main} cannot be improved from ET0L to context-free languages. 

Recall that a group $G\leq \AutTd$ has an induced action on the boundary of the rooted regular tree. We can thus consider the family of Schreier graphs associated with this action, as in \cref{sec:introduction}.
Recall also that $\WP(G,X,\Stab(\eta))$ is the language of words that label closed paths from $\eta$ to $\eta$ in the (rooted) Schreier graph $\Gamma_\eta$ of $\Stab(\eta)$.

\begin{definition}
Let $\Gamma_{v_{0}}$ be a labelled graph rooted in $v_{0}$ and let $v$ be a vertex of $\Gamma_{v_{0}}$.
The \emph{end-cone} $\Gamma_{v_{0}} (v)$ is the connected component of $\Gamma_{v_{0}} \setminus B_{\Gamma_{v_{0}}}(v_{0},|v|)$ which contains $v$, where $B_{\Gamma_{v_{0}}}(v_{0},k)$ denotes the (open) ball of radius $k$ centred at $v_{0}$ and $|v|$ is the distance from $v_0$ to $v$.
We denote by $\Delta_{v_{0}} (v)$ the set of vertices of $\Gamma_{v_{0}} (v)$ that are at a minimal distance from $v_{0}$, and we call them \emph{frontier points} of the end-cone.

We say that two end-cones $\Gamma_{v_{0}} (v_1)$ and $\Gamma_{v_{0}} (v_2)$ \emph{have the same type} if there exists a graph isomorphism $\varphi\colon \Gamma_{v_{0}} (v_1)\to\Gamma_{v_{0}} (v_2)$ between them that respects the labelling and for which $\varphi(\Delta_{v_{0}}(v_1))=\Delta_{v_{0}} (v_2)$.
\end{definition}

\begin{definition}
 A rooted graph is context-free if it has finitely many types of end-cones.
\end{definition}

On one hand, if $\WP(G,X,\Stab(\eta))$ is context-free, then so is the corresponding Schreier graph. On the other hand, context-free graphs are quasi-isometric to trees (see \cite[Propositions 7 and 9]{Rodaro2023}). Hence, we have our first criterion, as follows.

\begin{theorem}[see \cite{Rodaro2023}]\label{thm:Sch-not-CF}
Let $G$ be a finitely generated bounded automata group and let $\eta$ be an infinite ray in $\partial \Td$. If $\WP(G,X,\Stab(\eta))$ is context-free, then the Schreier graph $\Gamma_\eta$ is quasi-isometric to a tree.
\end{theorem}

We will now look at the number of ends, which is known to be a quasi-isometric invariant. It is proven in \cite[Corollary 5]{BDN2017}  that  Schreier graphs $(G,X,\Stab(\eta))$ of a bounded automata self-similar group $G$ have either almost surely one end or almost surely two ends. Almost surely here means for almost all $\eta \in \partial \Td$ with respect to 
 the uniform measure on $\partial\Td$. 
Let us first discuss the case of one end. It is straightforward that a one-ended tree is quasi-isometric to either the half line or the half line with infinitely many finite paths of unbounded lengths attached. For the first option, we will prove that a self-similar bounded automata group which is level-transitive cannot have more than two Schreier graphs quasi-isometric to a half-line. For the second option, it is not hard to see that the graph does not have finitely many end-cone types.

 In what follows, we will use results and methods from  \cite{Bondarenko2007thesis, BDN2017}. In particular, we require  our automata groups to be self-similar.


It will be helpful to us to describe the Schreier graph $\Gamma_{\xi}$, with $\xi = a_1a_2\cdots$, $a_i\in C$, in terms of the finite Schreier graphs  $(G,X,\Stab(a_{1} a_{2}\cdots a_{n}))$ that correspond to the stabilisers of the vertices of the tree $\Td$ that lie on the infinite ray $\xi$. If the action of $G$ on $\Td$ is transitive on every level $C^n$, then the Schreier graphs associated to $(G,X,\Stab(a_{1} a_{2}\cdots a_{n}))$ and $(G,X,\Stab(\widetilde{a}_{1} \widetilde{a}_{2}\cdots \widetilde{a}_{n}))$ are isomorphic as unrooted graphs; we will therefore denote such a graph simply by $\Gamma^n$. The vertices of $\Gamma^n$ are exactly the vertices of the $n$-th level, $C^n$, and  two vertices $a_{1} a_{2}\cdots a_{n}$ and $\widetilde{a}_{1} \widetilde{a}_{2}\cdots \widetilde{a}_{n}$ are joined by an edge whenever there exists an element of $X$ sending one to the other. By \cref{eq:stab} in the Introduction, the sequence of rooted graphs $\{(\Gamma^{n}, a_{1} a_{2}\cdots a_{n})\}_n$ converges to the rooted graph $(\Gamma_{\xi},\xi)$ in local topology. This means that for every radius $r$, the ball $B_{\Gamma_{\xi}}(\xi, r)$ in $\Gamma_{\xi}$ is isomorphic to the ball $B_{\Gamma^n}(a_{1} a_{2}\cdots a_{n}, r)$ in $\Gamma^{n}$, provided $n$ is sufficiently large.

\begin{theorem} \label{thm: Sch-not-half-line}
Let $G$ be a finitely generated self-similar bounded automata group acting transitively on any $C^n$, then the set 
\[
\{\mu\mid \Gamma_\mu \textrm{ is quasi-isometric to a half-line}\}
\] 
consists of at most two orbits.    
\end{theorem}

\begin{proof}
Suppose, by contradiction, that there are at least three Schreier graphs quasi-isometric to a half-line, say corresponding to the orbits of $\xi,\eta, \phi$.

Let $\xi=a_1a_2\cdots$ as before,  $\eta=y_1y_2\cdots$ and $\phi=z_1z_2\cdots$, and denote by $d(-,-)$ the geodesic distance in $\Gamma^n$ or in $\Gamma_{\xi}$. Since $\xi$, $\eta$ and $\phi$ lie in different orbits, the distances between their prefixes must diverge as $n\to\infty$.

For each $n$, consider the vertices
$$
\xi_{y,n}:=y_1y_2\cdots y_na_{n+1}a_{n+2}\cdots \quad\text{and}\quad \xi_{z,n}:=z_1z_2\cdots z_na_{n+1}a_{n+2}\cdots.
$$
These vertices belong to $\Gamma_\xi$ for all $n$, since they are cofinal with $\xi$ (see \cite{BDN2017}). By the definition of convergence above, and using the divergence of prefixes discussed in the previous paragraph, we have
$$
d(\xi,\xi_{y,n}) \rightarrow \infty,\qquad d(\xi,\xi_{z,n}) \rightarrow \infty,\qquad d(\xi_{y,n},\xi_{z,n}) \rightarrow \infty \quad \text{as } n \rightarrow \infty,
$$
and the subgraphs induced by $\{\xi_{y,n}\}_{n\in\N}$ and $\{\xi_{z,n}\}_{n\in\N}$ inside $\Gamma_\xi$ are quasi-isometric to a half-line. Consequently, $\Gamma_\xi$ must have at least two ends, which yields a contradiction.
 \end{proof}

Self-similar bounded automata groups for which almost all Schreier graphs have two ends are listed in \cite{BDN2017}. In the case of binary alphabet $X$, these correspond to automata that appear in \cite{Sunic2007}, including the first Grigorchuk group.
To it, we can apply \cref{prop:torsion-not-CF} which applies, more generally, to any bounded automata torsion group. To start, we need the following lemma. 

\begin{lemma}\label{lem:regular-geodesics}
Let $G$ be a finitely generated bounded automata group and let $\eta$ be an infinite ray in $\partial \Td$. If the Schreier graph $\Gamma_\eta$ is context-free, then the set of words one can read on geodesics in the graph starting from $\eta$ is a regular language.
\end{lemma}

\begin{proof}
We recall that, by definition, $\Gamma_\eta$ has finitely many end-cone types. All we have to do is to construct a finite state automaton that reads geodesics. The states are of the form $(C,\zeta)$ where $C$ is an end-cone type and $\zeta$ ranges in the frontier points of a given end-cone of type $C$. Note that the set of states is finite since the number of end-cone types is finite, and the number of frontier points for a given end-cone is finite too.
The initial state is $(C_{0},\eta)$ with $C_{0}$ the end-cone type of the base vertex $\eta$ and all the states are final. We then add a transition $(C_1,\zeta_1)\xrightarrow{a} (C_2,\zeta_2)$ if there exists an edge, labelled with $a$, from the frontier point corresponding to $\zeta_1$ of an end-cone of type $C_1$, to a point that corresponds to $\zeta_2$ which belongs to $\Gamma_\eta (\zeta_1) \setminus \Delta_{\eta} (\zeta_1)$.
Moreover, this vertex $\zeta_2$ is a frontier point of an end-cone of type $C_2$. 
\end{proof}

\begin{proposition}\label{prop:torsion-not-CF}
Let $G$ be an infinite finitely generated torsion bounded automata group, and let $\eta$ be an infinite ray. If the Schreier graph of $\eta$ is infinite, then $\WP(G,X,\Stab(\eta))$ is not context-free.
\end{proposition}
\begin{proof}
If $\WP(G,X,\Stab(\eta))$ is context-free, then so is the corresponding Schreier graph. Thus, from \cref{lem:regular-geodesics} we know that the language of all geodesics in this graph is regular.
Applying the pumping lemma for regular languages (see, e.g.\@ \cite[Theorem 1.70]{S2013}), we see that this language of geodesics contains some sub-language $\{xy^nz\mid n\in \mathbb N \}$  where $y \in X^*$ is a non-empty word.
This contradicts our assumption that $G$ is torsion, since $xy^kz$ cannot be a geodesic when $k$ is the order of the element given by $y$.
\end{proof}

To summarise, combining \cref{thm:Sch-not-CF} with \cref{thm: Sch-not-half-line} we show that there are finitely generated bounded automata groups such that almost surely the Schreier graphs are not context-free. An interesting example of a group where almost all Schreier graphs are one-ended trees but are not context-free, is the following.

\medskip
\noindent\textit{\underline{Iterated monodromy group of $z^2+i$}} (see \cite{BDN2017}). Using the notation introduced in \cref{sec:bounded-automata-groups}, the group is generated by the three automorphisms
$$ a=(b,c),\ \ b=(1,1) \cdot s,\ \ c=(a,1)$$
where $1 \neq s \in \mathrm{Sym}(\{0,1\})$.
In this case, it is easy to see that the trees are not quasi-isometric to half-lines. Indeed, 
if $\eta=c_1c_2\cdots$ is an infinite ray, the graph $\Gamma_{\eta}$ contains vertices $v_k:=0^kc_{k+1}\cdots$. By an inductive argument, one can show that from each $v_k$ there is a path to $1^kc_{k+1}\cdots$ and a  different path to $1^{k-2}00c_{k+1}\cdots$. In particular, if the graph is one-ended, then it is a half-line with infinitely many paths of unbounded lengths attached, and it is easy to see that such a graph does not have finitely many end-cone types, and so is not context-free.

\medskip

Our analysis above is based on the number of ends in a typical Schreier graph, that is, a Schreier graph from a subset of $\partial \Td$ of measure zero. But our \cref{thm:main} concerns infinite rays that are periodic, and such rays form a subset of $\partial \Td$ of measure zero. Therefore, a stronger version of the theorem where  ``ET0L'' would be replaced with ``context-free'' might still be possible. Below, we will use \cref{thm:Sch-not-CF} and \cref{prop:torsion-not-CF} to provide some examples where this is not the case, as all (and not only almost all) the Schreier graphs are not context-free.
A useful result here is Theorem~11~in~\cite{BDN2017} which provides a criterion to determine whether all the Schreier graphs are one-ended.

\medskip
\noindent\textit{\underline{Hanoi tower group on three pegs}} (see \cite{GS2006}). It is known that any Schreier graph of this group is one-ended and not quasi-isometric to the half-line (e.g.~Remark~3~in~\cite{BDN2017}), hence any $\WP(G,X,\Stab(\eta))$ is not context-free.

\medskip
\noindent\textit{\underline{Basilica group}}.
We can directly apply the theorem to show that the language $\WP(G,X,\Stab(\eta))$ cannot be context-free for the Basilica group, since
its Schreier graphs are fully classified in \cite{DDMN2010}. Namely, all Schreier graphs are one-ended, two-ended or four-ended. In fact, the latter case is a single exception.
If the Schreier graph has one or two ends, by the classification it is not quasi-isometric to a tree. 
On the other hand, for the case of the four-ended graph, it is clear that it does not have finitely many end-cone types (see \cite[Theorem 4.6 and Figure 7]{DDMN2010}).

\medskip
\noindent\textit{\underline{First Grigorchuk group}}. In this example almost all graphs are two-ended graphs. Namely, it has just one one-ended Schreier graph,  the one containing $1^{\infty}$, the rightmost point in the boundary. 
In this case, we use \cref{prop:torsion-not-CF} to conclude the non-context-freeness.

\medskip

We end the section with two examples. One, for which \cref{thm:main} can indeed be  strengthened, and $\WP(G,X,Stab(\eta))$ is context-free. And one where we do not know whether $\WP(G,X,\Stab(\eta))$ is context-free.

\medskip
\noindent\textit{\underline{Infinite dihedral group}}. This group can be seen as the self-similar group generated by the automorphisms $$ a=(1,1) \cdot s,\ \ b=(a,b),$$
where $1 \neq s \in \mathrm{Sym}(\{0,1\})$ or, equivalently by the bounded automaton in \cref{fig:dihedral-automaton}.
One Schreier graph is one-ended (containing $1^{\infty}$), while all the others are isomorphic to the Cayley graph of $\mathbb{Z}$, an infinite line, see \cref{fig:dihedral-schreier}. Thus, all the Schreier graphs are context-free. 

\medskip

\begin{figure}[!ht]
	\centering	
\begin{tikzpicture}[->,>=stealth,shorten >=1pt,auto,node distance=3.5cm,on grid]
   \node[state] (a) [] {$a$}; 
   \node[state] (b) [left=of a]  {$b$}; 
   \node[state] (1) [right=of a]  {$1$}; 
       \path[->] 
       (a) edge [bend left, above] node {$(0,1)$} (1) 
           edge [bend right, below] node {$(1,0)$} (1);
       \path[->]  
       (b)  edge node {$(0,0)$} (a)
            edge [loop left] node {$(1,1)$} (); 
        \draw (1) to [out=60,in=30,looseness=8] node [above] {$(0,0)$} (1);
        \draw (1) to [out=330,in=300,looseness=8] node [below] {$(1,1)$} (1);
\end{tikzpicture}
\caption{Automaton of the infinite dihedral group.}
	\label{fig:dihedral-automaton}
\end{figure}
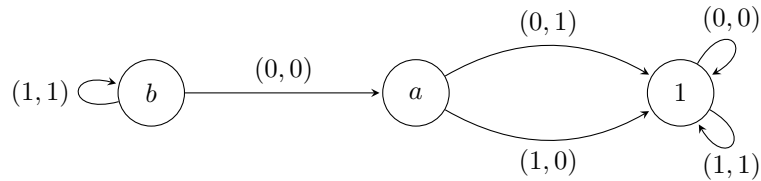
\begin{figure}[!ht]
	\centering	
 \begin{tikzpicture}
 
 \begin{scope}[every loop/.style={min distance=12mm}]
\draw (-3,0) -- node[above]{$a$} (-2,0)  -- node[above]{$b$} (-1,0) -- node[above]{$a$} (0,0) -- node[above]{$b$}  (1,0) -- node[above]{$a$} (2,0) -- node[above]{$b$} (3,0) ;

\node at (-3,0) () [above]{$1^{\infty}$};
\draw(-3,0) edge[loop left,left] node {$b$} ();
\draw[dashed] (3,0) -- (4,0);

\fill (-3,0) circle (0.35ex);
\fill (-2,0) circle (0.35ex);
\fill (-1,0) circle (0.35ex);
\fill (0,0) circle (0.35ex);
\fill (1,0) circle (0.35ex);
\fill (2,0) circle (0.35ex);
\fill (3,0) circle (0.35ex);
\end{scope}

\begin{scope}[yshift=-40]
\draw (-3,0) -- node[above]{$a$} (-2,0)  -- node[above]{$b$} (-1,0) -- node[above]{$a$} (0,0) -- node[above]{$b$}  (1,0) -- node[above]{$a$} (2,0) -- node[above]{$b$} (3,0) ;

\node at (-3,0) () [above left]{$\eta$};
\draw[dashed] (-4,0) -- (-3,0);
\draw[dashed] (3,0) -- (4,0);

\fill (-3,0) circle (0.35ex);
\fill (-2,0) circle (0.35ex);
\fill (-1,0) circle (0.35ex);
\fill (0,0) circle (0.35ex);
\fill (1,0) circle (0.35ex);
\fill (2,0) circle (0.35ex);
\fill (3,0) circle (0.35ex);
\end{scope}
\end{tikzpicture}
\caption{Isomorphism classes of Schreier graphs in the infinite dihedral group.}
	\label{fig:dihedral-schreier}
\end{figure}
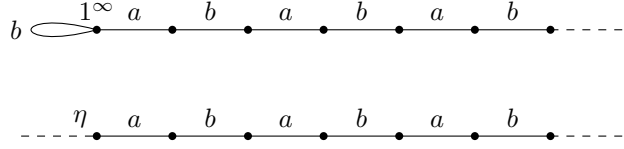

\medskip
\noindent\textit{\underline{Grigorchuk group $G_{\overline{01}}$}} (see, e.g.~\cite{grigorchuk1984}). It is well known that this is a self-similar bounded automata group containing non-torsion elements and all the Schreier graphs are quasi-isometric to a line or a half-line. However, we do not know whether such graphs are context-free or not.

\section{Further Research} \label{sec:future}

In this paper we showed that in a finitely generated group $G\leq \AutTd$ the membership problem $\WP(G,X,\Stab(\eta))$ in stabilisers of infinite eventually periodic rays $\eta\in \partial \mathcal{T}_d$ is an ET0L language. It turns out that the word problem $\WP(G,X)$
can be characterised in terms of these languages, as explained in the following proposition. 

\begin{proposition}
 The word problem of a group $G \leq \AutTd$ coincides with the intersection of all $\WP(G,X,\Stab(\eta))$ with $\eta$ periodic.   
\end{proposition}\label{lem:stab-wp}

\begin{proof}
It is straightforward that any $\WP(G,X,\Stab(\eta))$ contains the word problem. On the other side, the action on $\mathcal{T}_{d}$ is faithful. So, if an element stabilises all the vertices of the tree, then it is the identity. 
Now observe that if $w \in \WP(G,X,\Stab(\eta))$, then $w$ stabilises all the prefixes of $\eta$.
Take  $w$ in the intersection of all $\WP(G,X,\Stab(\eta))$ with $\eta$ periodic. This means that $w$ stabilises all possible finite words in $C^*$ and hence it is the identity.
\end{proof}

The immediate corollary that the word problem of a bounded automata group is an intersection of infinitely many ET0L languages does not in itself say much, as it is well known that any language is an intersection of (infinitely many) regular languages. But it motivates the following natural question.

\begin{question}
Is it true that the word problem of a bounded automata group is a finite intersection of ET0L languages?
\end{question}

In \cref{sec:not-CF}, we proved that the languages $\WP(G,X,\Stab(\eta))$ are not context-free under some additional hypothesis. We also mentioned there one  example, a non-torsion group from Grigorchuk's family, $G_{\overline{01}}$, for which we do not know whether these languages are context-free or not. We think, it is not, and we ask the following.

\begin{question}
Is there  a non-virtually free self-similar bounded automata group with $\WP(G,X,\Stab(\eta))$ context-free?
\end{question}

We are also interested to know if $\eta$ is computable from $\WP(G,X,\Stab(\eta))$. That is, if the language uniquely determines the ray.

\begin{question} Let $G$ be the first Grigorchuk group, and let  $\eta$ be the word $0101^2 01^3 01^4 \ldots$. Is it true that the subgroup membership problem $\WP(G,X,\Stab(\eta))$ is not ET0L?
\end{question}

The statements \cref{thm:gfun,thm:main} put together give us a characterisation of the generating function for $\WP(G,X,\Stab(\eta))$, when $\eta$ is an  eventually periodic infinite word in the alphabet $C$.
There are certain subclasses of \emph{indexed languages} that have known characterisations of their generating functions with potential closed-form expressions (see~\cite{adams2013}).

\begin{question}
    For what bounded automata group and infinite rays does the language $\WP(G,X,\Stab(\eta))$ belong to the subclasses of indexed languages as studied in~\cite{adams2013}?
\end{question}

\section*{Acknowledgements}
The first-named author and the fifth-named author acknowledge support from the Swiss Government Excellence Scholarship.
The first-, fourth- and fifth-named authors acknowledge support from Swiss NSF grant 200020-200400.
The second-, third-, fifth- and sixth-named authors are members of the Gruppo Nazionale per le Strutture Algebriche, Geometriche e le loro Applicazioni (GNSAGA) of the Istituto Nazionale di Alta Matematica (INdAM). 
The third-named author is also a member of the PRIN 2022 ``Group theory and its applications'' research group and gratefully acknowledges the support of the PRIN project 2022-NAZ-0286, funded by the European Union - Next Generation EU, Missione 4 Componente 1 CUP B53D23009410006, PRIN 2022 - 2022PSTWLB - Group Theory and Applications. The third-named author also gratefully acknowledges the support of the Universit\`a degli Studi di Milano--Bicocca
(FA project 2021-ATE-0033 ``Strutture Algebriche'').
The fifth-named author acknowledges support from the Grant QUALIFICA by Junta de Andalucía grant number QUAL21 005 USE and from the research grant PID2022-138719NA-I00 (Proyectos de Generación de Conocimiento 2022) financed by
the Spanish Ministry of Science and Innovation.
The first-named author thanks the Dipartimento di Matematica e Applicazioni of the Universit\`a di Milano-Bicocca for their hospitality.
The fifth-named author also thanks the Section de mathématiques of the
Université de Genève for their hospitality.
We are grateful to Murray Elder for helpful discussions and to the anonymous referees for useful comments on the first draft of the paper.

\bibliographystyle{plain}
\bibliography{main}

\end{document}